\documentclass[12pt]{article}

\usepackage[margin=1in]{geometry}
\usepackage{amsmath,amssymb,amsthm}
\usepackage{enumitem}
    \setlist{itemsep=0mm}
\usepackage{pifont}
\usepackage{csquotes}

\usepackage{tikz}
\usetikzlibrary{positioning,fit}
\usepackage{caption}
\usepackage{graphicx}

\usepackage[citestyle=authoryear,bibstyle=authoryear]{biblatex}
\newtheorem{theorem}{Theorem}

\newtheorem{corollary}{Corollary}
\newtheorem{definition}{Definition}

\newtheorem{proposition}{Proposition}
\newtheorem{claim}{Claim}
\newtheorem{remark}{Remark}

\DeclareSymbolFont{symbolsC}{U}{txsyc}{m}{n}
\SetSymbolFont{symbolsC}{bold}{U}{txsyc}{bx}{n}
\DeclareMathSymbol{\coloneqq}{\mathrel}{symbolsC}{"42}
\DeclareMathSymbol{\Coloneqq}{\mathrel}{symbolsC}{"46}

\newcommand{\xmark}{\ding{55}}
\newcommand{\mystar}{\mathord{\star}}

\newcommand{\tup}[1]{\langle #1 \rangle}

\newcommand{\would}{\mathbin{>}}
\newcommand{\nec}{\mathop{\Box}}
\newcommand{\pos}{\mathop{\Diamond}}
\newcommand{\then}{\mathbin{\supset}}
\newcommand{\ex}{\mathop{\textit{E}}}

\title{Incompleteness in Quantified Conditional Logic}
\author{Alexander W.\ Kocurek, James Walsh, Yale Weiss}

\date{}

\begin{document}
\maketitle

\begin{abstract}
    Stalnaker and Thomason famously proved that the conditional logic \textsf{C2} with first-order quantifiers is complete with respect to a selection function semantics. 
    However, the selection functions used in this completeness result take formulas, rather than propositions (i.e., sets of worlds), as arguments. 
    Yet Stalnaker has repeatedly emphasized the philosophical importance of viewing selection functions as functions on propositions, and many of the applications of his theory require this. 
    Can their completeness result be extended to a selection function semantics in which the functions take propositions as arguments? 
    We prove the answer is negative: Their logic is frame incomplete. 
    Moreover, this result is invariant with respect to many choice points regarding the semantics, such as variable vs.~constant domains or whether to include an identity or existence predicate. We conclude by discussing some of the important and difficult questions for the philosophical and logical study of conditionals that our results raise.

    \bigskip\noindent\textbf{Keywords.} Conditionals; Incompleteness; Quantifiers; Selection functions; Propositions.
\end{abstract}

\section{Introduction}
\label{Section:Introduction}

In \emph{A Theory of Conditionals}, Stalnaker writes:
\begin{displayquote}[\cite{Stalnaker1968theory}, p.~98]
    My principal concern will be with what has been called the \emph{logical problem of conditionals}\dots 
    This is the task of describing the formal properties of the \emph{conditional function}: a function, usually represented in English by the words `if\dots then', taking ordered pairs of propositions into propositions.
\end{displayquote}
Stalnaker identifies propositions with \emph{sets of possible worlds} and explicates the conditional function in terms of the \emph{selection function}, which maps a given proposition $P$ and world $w$ to the ``closest'' $P$-world to $w$. 
Thus, as Stalnaker conceived of it, any solution to the logical problem of conditionals requires axiomatizing the logical behavior of selection functions so-understood.

Stalnaker claims (ibid, p.\ 98, fn.\ 1) that the propositional version of this problem is answered by a completeness theorem for a logic called \textsf{C2}, citing an early version of \cite{StalnakerThomason1970}. 
The published version goes further and addresses a first-order variant of the logical problem of conditionals.
They prove a completeness theorem for a logic, which we will call \textsf{QST} (Quantified Stalnaker-Thomason logic), with the goal of establishing that Stalnaker's theory of conditionals is ``technically respectable'' and satisfies ``the standards currently set by logicians'' \parencite[41]{StalnakerThomason1970}.\footnote{What we call \textsf{QST} is what Stalnaker and Thomason call \textsf{CQ}. We use the label $\mathsf{QST}$ to avoid confusion with other labels we introduce later (e.g., $\mathsf{QC2}$).}

While it is debatable whether Stalnaker and Thomason presented the \emph{correct} answer to the logical problem of conditionals, it is generally thought that they presented \emph{an} answer to the problem. 
The aim of this paper is to reassess this claim. 
We argue that, in fact, the logic $\mathsf{QST}$ does not---indeed, \emph{cannot}---constitute a solution to the logical problem of conditionals as Stalnaker conceived it. 

To see why, it is important to emphasize that Stalnaker characterizes the logical problem of conditionals in terms of propositions rather than linguistic expressions. 
This characterization is not incidental but is rooted in the theoretical roles that both propositions and conditionals play in his work.
Regarding propositions, Stalnaker has repeatedly emphasized that propositions are not linguistic items and that any ``account of propositions that treats them as linguistic items of some kind would be inappropriate'' \parencite[pp.\ 82--83]{stalnaker1976propositions}. For instance, Stalnaker accounts for rational action in terms of beliefs and desires, where the objects of beliefs and desires are propositions; since an account of rational action should apply to both linguistic and non-linguistic creatures, propositions must be understood non-linguistically. Regarding conditionals, Stalnaker has used his theory of conditionals to analyze, among other things, rational action, communication, laws of nature, and the probabilities of conditionals. Since these are understood in terms of propositions, conditionals must be understood as functions on propositions.


There is also a considerable amount of work within neo-Stalnakerian frameworks that crucially understands the conditional function by means of selection functions \emph{on propositions}. 
For instance, in contemporary neo-Stalnakerian work on the probabilities of conditionals (see, e.g., \cite{Schulz2014,Schulz2017,Bacon2015,GoldsteinSantorio2021,Khoo2022,KhooMandelkern2019,Santorio2022,Santorio2025,Schultheis2023}), probability functions---whether they are understood in terms of credences or objective chances---are treated as functions on non-linguistic entities, as are selection functions.
\footnote{
    We say ``non-linguistic entities'' rather than ``propositions'' since some of these analyses use probability functions on more fine-grained entities than sets of worlds (e.g., sets of sequences of worlds). 
    Occasionally, these authors will talk of the probability of a ``sentence'', but this is always understood to mean the probability of the proposition(-like entity) expressed by a sentence. 
    } 
This non-linguistic orientation to conditional credences is in keeping with Stalnaker's claim that non-linguistic creatures can also have credences \parencite{stalnaker1976propositions}. Similarly, Stalnaker's theory and logic of conditionals is often invoked in contemporary decision theory, specifically in the context of causal decision theory, where expected utility invokes the probabilities of certain counterfactual conditionals \parencite{GibbardHarper1978,Lewis1981,Joyce1999,Hedden2023,McNamara2024}.  
Here again, it is crucial that the probability of a conditional is not the probability of a sentence necessarily, but rather the probability of a conditional proposition, since decision theory is meant to provide a theory of rationality for any agent, including non-linguistic ones. 




Thus, the decision to characterize the logical problem of conditionals in terms of propositions rather than linguistic expressions is not arbitrary, but rather is central to the intended applications of the theory. Indeed, Stalnaker and Thomason discuss potential applications of their completeness theorem for \textsf{QST}, and it is clear from their discussion that they understand this completeness theorem to have addressed the status of the logical problem \emph{for propositions}. 
For example, Stalnaker and Thomason explore an analysis of laws of nature in terms of conditionals that appeals directly to properties of $\mathsf{QST}$. 
As they observe, a linguistic analysis of laws would not be plausible, as nothing requires the laws of nature to be stated or even stateable in one's object language: 
\begin{displayquote}
    One of the advantages of the semantic approach is that it allows us to focus directly on the world rather than on linguistic expressions purporting to describe it. 
    One might, for example, give an analysis of explanation in terms of possible worlds and s-functions rather than indirectly in terms of laws. 
    This way, one could avoid the embarrassment of unstated or unknown laws, and one might find an analysis which handles more plausibly the everyday cases of explanation. \parencite[p.\ 41]{StalnakerThomason1970}
\end{displayquote}
Likewise, in discussing a definition of probabilities of conditionals in terms of conditional probabilities, 
Stalnaker and Thomason explicitly invoke the logical properties of $\mathsf{QST}$ in their analysis, writing:
\begin{displayquote}
    If this definition is extended in a natural way, the resulting conditional concept has exactly the structure of the conditional defined by the system [\textsf{QST}]. 
\parencite[p.\ 41]{StalnakerThomason1970}
\end{displayquote}

And yet, the selection functions with respect to which Stalnaker and Thomason obtain their completeness results are not functions on \emph{propositions}. That is, the completeness result does \emph{not} axiomatize the behavior of the functions that \cite{Stalnaker1968theory} takes himself to be concerned with. 
Rather, the selection functions used to prove completeness are functions on \emph{formulas} of the object language.
\footnote{
    See D3.4 in \cite[p.~27]{StalnakerThomason1970}. 
    In effect, they prove completeness with respect to a class of \emph{general frames} (cf.~\cite{Segerberg1989}).
    } 
In this framework, in other words, semantic values are not attributed to all conditional propositions by means of selection functions, but only to those whose antecedents are expressible. 
Thus, the Stalnaker--Thomason logic does not achieve the goal---which has been reiterated in various places (see, e.g., \cite[Ch.~7]{Stalnaker1984})---of characterizing the formal behavior of selection functions on \emph{propositions}. 

This leaves an important theoretical question about the logic of conditionals unresolved: 
Can $\mathsf{QST}$ be characterized by a semantics that realizes Stalnaker's guiding intuition? 
To put the matter precisely: Is $\mathsf{QST}$ the logic of any class of selection function frames, wherein functions are defined for \textit{all} propositions and worlds? Our main theorem is that the Stalnaker--Thomason logic $\mathsf{QST}$ is \emph{frame incomplete}, that is, it cannot be characterized by any such class of frames. 
It is thus \emph{impossible} for $\mathsf{QST}$ to constitute an answer the logical problem of conditionals as \cite{Stalnaker1968theory} articulates it.


In fact, we prove something even stronger. 
The original Stalnaker--Thomason logic \textsf{QST} is stated in a language that includes the identity symbol. 
It does not include a primitive existence predicate, but an existence predicate is definable in terms of identity. 
Moreover, \textsf{QST} is intended to capture actualist quantification, that is, the intended semantics has variable domains of quantification. 
As we will show, our incompleteness theorem does not depend on any of these choice points. 
To illustrate this, we concentrate on a system we call \textsf{QC2}, which is the simplest among the systems we consider. 
This system is essentially the result of combining \textsf{C2} with the standard axioms and rules governing first-order quantifiers in the most straightforward manner. 
The language of \textsf{QC2} lacks symbols for identity and existence and the intended semantics for \textsf{QC2} has a single \emph{global} domain of quantification. 
After proving our incompleteness theorem for \textsf{QC2}, we explain how to modify it to accommodate the aforementioned choice points, e.g., to include an existence or identity predicate or to allow variable domains of quantification.


Although our result appears to be the first result of its kind in quantified conditional logic,
    \footnote{\cite{Nute1978} presents an incompleteness result for propositional conditional logic (which, however, piggy-backs on incompleteness results in propositional modal logic) with respect to the sphere semantics of \cite{Lewis1973}.
    } 
the frame incompleteness phenomenon is well-known in modal logic. 
Most propositional normal modal logics that philosophers and logicians have been interested in can be characterized by Kripke frames by playing with the conditions on the binary accessibility relation. 
However, a number of  examples of propositional normal modal systems that are not the logics of any class of Kripke frames (i.e., \emph{frame incomplete} modal logics) have been found since the mid-1970s.
\footnote{
    See \cite{Fine1974incomplete}, \cite{Thomason1974incomplete}, \cite{vanBenthem1978incomplete}, and \cite{BoolosSambin1985incomplete}, as well as textbook presentations of such systems in \cite[Ch.~9]{HughesCresswell1996} and \cite[\S4.4]{Blackburn2001}. 
    This does not mean that such systems are without \emph{any} semantics; they can be characterized algebraically, or using general frames (i.e., Kripke frames with an attached algebra of propositions into which admissible valuations map propositional variables). 
    For some discussion, see \cite[Ch.~5]{Blackburn2001}.
    } 
While incomplete propositional modal logics tend to be a bit artificial,
\footnote{
    More precisely, incomplete propositional \emph{monomodal} logics tend to be artificial. Fairly natural examples of incomplete propositional \emph{bimodal} logics include the tense logic of \cite[\S4]{Thomason1972tense} and the bimodal provability logic $\mathsf{GLB}$ (see \cite[pp.~193--194]{Boolos1993provability}).
    } 
there are natural examples of frame complete propositional modal logics whose extensions with quantifiers are frame incomplete. 
A particularly nice (and suggestive) example is the provability logic $\mathsf{GL}$. 
On the one hand, propositional $\mathsf{GL}$ is frame complete---it is the logic of finite Kripke frames in which the accessibility relation is transitive and converse well-founded---but it is not strongly complete, i.e., there are some (infinitary) entailments that cannot be derived in $\mathsf{GL}$.
\footnote{
    The (weak) completeness of $\mathsf{GL}$ (alias: $\mathsf{K4W}$) was proved by \cite[pp.~84--89]{Segerberg1971Dissertation}; for a textbook presentation, see \cite[Chs.~4--5]{Boolos1993provability}. For a proof that $\mathsf{GL}$ is not strongly complete, see \cite[p.~211]{Blackburn2001}. The reader may find it illuminating to compare this argument with the argument \cite{kocurek2025stalnaker} gives to show that strong completeness fails for $\mathsf{C2}$.
    } 
On the other, \cite{Montagna1984provability} proved that $\mathsf{GL}$ with quantifiers (but not the Barcan formula) is frame incomplete.
\footnote{
    For many related results, see \cite{Rybakov2024incompleteness}.
    } 
As we shall see, the situation with Stalnaker's conditional logic turns out to be rather similar. 
In particular, while $\mathsf{C2}$ is weakly complete with respect to the propositional selection function semantics (i.e., all validities are $\mathsf{C2}$-theorems),\footnote{This result is not proved in \cite{Stalnaker1968theory}, but a proof sketch can be found, for example, in \cite[\S6]{vanFraassen1974hiddenvariables}.}
it is not strongly complete with respect to this semantics, i.e., some (infinitary) entailments are not derivable in $\mathsf{C2}$.
\footnote{
    Failure of strong completeness for $\mathsf{C2}$ was observed or proved independently by \cite[\S6]{vanFraassen1974hiddenvariables}, \cite[p.~105]{Veltman1985logics}, \cite[fn.\ 2]{fine2012difficulty}, \cite{kocurek2025stalnaker}, and \cite{dorr2024logic}.
    } 
Our incompleteness theorem effectively fills in the analogy with $\mathsf{GL}$ by showing that $\mathsf{C2}$ extended with quantifiers (be they actualist or possibilist, with or without the Barcan formula or its converse) is frame incomplete. 

The plan of our paper is as follows. In Section~\ref{Section:Semantics}, we survey the (set) selection function semantics as well as an alternative ordering semantics which we will make technical use of. 
Axiom systems for various conditional logics we are interested in are presented in Section~\ref{Section:Logics}, where we also discuss known completeness results. 
In Section~\ref{Section:Proof of Incompleteness}, we prove the frame incompleteness of \textsf{QC2}. 
We demonstrate how to extend the result to other quantified conditional logics, including Stalnaker and Thomason's \textsf{QST}, in Section~\ref{Section:Refinements}. 
Finally, we offer some concluding remarks in Section~\ref{Section:Concluding Remarks}.

\section{Semantics}
\label{Section:Semantics}


In this section, we survey two kinds of semantics (or model theory) for quantified conditional logics. 
In Section~\ref{Subsection:Set Selection Function Semantics}, we examine the (set) selection function semantics with respect to which we ultimately prove incompleteness. 
In Section~\ref{Subsection:Ordering Semantics}, we examine an alternative ordering semantics, which we use to establish an independence result essential to our incompleteness proof.

Before turning to semantics proper, we fix some details concerning the formal languages we will work with. 

\begin{definition}
\label{def:lang}
    We assume that we have a countable stock of individual variables $(x_i)_{i<\omega}$ and, for each $n$, countably many $n$-place predicates $(P_i^n)_{i<\omega}$. The language $\mathcal{L}$ of quantified conditional logic is given below:
    \begin{align*}
        \phi & \Coloneqq P^n (x_1,\dots,x_n) \mid \neg\phi \mid (\phi \then \phi) \mid (\phi \would \phi) \mid \forall x\phi
    \end{align*}
    where $P^n$ is an $n$-place predicate. 
    The other extensional operators---$\wedge$, $\vee$, $\equiv$, $\top$, $\bot$, and $\exists$---and defined in the standard way. 
    We also adopt the following standard abbreviations:
    \begin{align*}
        \nec\phi & \coloneqq (\neg\phi \would \bot) \\
        \pos\phi & \coloneqq \neg(\phi \would \bot)
    \end{align*}
    The language $\mathcal{L}_{\ex}$ is the result of adding a designated existence predicate $E$ to $\mathcal{L}$. 
    The language $\mathcal{L}_=$ is the result of adding an identity predicate $=$ to $\mathcal{L}$. 
    Though $\mathcal{L}_=$ technically lacks the existence predicate $E$, when working with $\mathcal{L}_=$, we use $E(x)$ as an abbreviation for the formula $\exists y(x = y)$. 
\end{definition}

We abuse notation in standard ways and often write things like $x$ and $y$ for individual variables and $P$ and $F$ for predicates, allowing arity to be determined by the context. 


\subsection{Selection Function Semantics}
\label{Subsection:Set Selection Function Semantics}

The main focus of this paper is on the selection function semantics introduced by \cite{Stalnaker1968theory} and \cite{StalnakerThomason1970}. 
Historically, this was presented as a \textit{world selection function} semantics in which the conditional is modeled using a function which maps pairs consisting of antecedents and worlds to individual worlds (intuitively, the closest antecedent-worlds to the world of evaluation). 
Here we present the semantics in an equivalent form, following \cite[\S2.7]{Lewis1973}, as a \textit{set selection function} semantics, which also allows us to dispense with certain technical devices (e.g., the impossible world $\lambda$).

\begin{definition}
\label{Definition:Set Selection Frame}
    A \emph{set selection function frame} (hereafter, \emph{selection frame}) is a quintuple $\mathcal{F} = \tup{W,R,f,D,d}$, where:
    \begin{itemize}
        \item $W \neq \varnothing$ is a set of worlds;
        \item $R \subseteq W \times W$ is the accessibility relation; 
        \item $f:\mathcal{P}(W)\times W \to \mathcal{P}(W)$ is the (set) selection function;
        \item $D \neq \varnothing$ is the (global) domain;
        \item $d\colon W \rightarrow \mathcal{P}(D)$ is the local domain assignment.\footnote{We follow \cite[p.~25]{StalnakerThomason1970} in allowing for empty local domains.}
    \end{itemize}
    We write $R(w)$ for $\{x \in W \mid wRx\}$. The only further condition we impose is that for any $P \subseteq W$ and any $w\in W$, $f(P,w) \subseteq R(w)$.
\end{definition}

Since $f$ is defined for all propositions ($P\subseteq{W}$) in Definition~\ref{Definition:Set Selection Frame}, these frames are \textit{full} in a technical sense (cf.~\cite[p.~161]{Segerberg1989}).

Throughout, we consider various constraints one might naturally impose on selection frames. 
One class of constraints concerning the local domain assignment, reflecting different interpretations of the first-order quantifier in modal settings (e.g., actualist vs.\ possibilist). 

\begin{definition}
\label{Definition:Domains}
    Let $\mathcal{F} = \tup{W,R,f,D,d}$ be a selection frame. 
    \begin{itemize}
        \item $\mathcal{F}$ is a \emph{(globally) constant domain frame} if $d(w) = D$ for all $w \in W$; 
        \item $\mathcal{F}$ is a \emph{locally non-decreasing frame} if $d(w) \subseteq d(v)$ when $wRv$;
        \item $\mathcal{F}$ is a \emph{locally non-increasing frame} if $d(v) \subseteq d(w)$ when $wRv$;
        \item $\mathcal{F}$ is a \emph{locally constant domain frame} if it is locally non-decreasing and non-increasing. 
    \end{itemize}
    For globally constant domain frames, we often omit mention of $d$, writing $\mathcal{F}=\tup{W,R,f,D}$. 
    We sometimes use the term \emph{``variable domain frame''} when talking about selection frames, regardless of the constraints on $d$. 
\end{definition}

Note, while being globally constant implies being locally constant, the reverse isn't true. 
For example, it may be that some elements of $D$ are ``impossibilia'', in that they do not appear in $d(w)$ or in $d(v)$ for any $v \in R^*(w)$, where $R^*$ is the transitive closure of $R$, even though $d(w) = d(v)$ for all such $v$. 
Such frames are thus locally but not globally constant. 
This difference manifests in different logical principles, as we will observe momentarily. 

Another class of constraints concerns the selection function. 
In particular, we focus on selection frames that obey the kinds of constraints that \cite{Stalnaker1968theory} himself imposes on selection functions (translated into set selection functions). 

\begin{definition}
\label{Definition:Stalnakerian Function}
    A selection frame $\tup{ W,R,f,D,d}$ is \emph{Stalnakerian} just in case it satisfies these conditions for all $P,Q\subseteq W$, $w\in W$:
    \begin{enumerate}
        \item $f(P,w) \subseteq P$ \hfill (Success)
        \item If $w \in P$, then $w \in f(P,w)$ \hfill (Weak Centering)
        \item If $f(P,w) = \varnothing$, then $P \cap R(w) = \varnothing$ \hfill (LA)
        \item If $f(P,w) \subseteq Q$ and $f(Q,w) \subseteq P$, then $f(P,w) = f(Q,w)$ \hfill (Uniformity)
        \item $|f(P,w)| \leq 1$ \hfill (Uniqueness)
    \end{enumerate}
    A selection frame is \emph{weakly Stalnakerian} if it satisfies conditions 1, 2, 4, and 5.
\end{definition}

\begin{remark}
\label{Remark:Elementary Stalnakerian Properties}
    Weak Centering implies that the accessibility relation in any (weakly) Stalnakerian frame is reflexive. 
    Moreover, in the presence of Uniqueness, Weak Centering is equivalent to ``Strong Centering'' ($w \in P$ implies $f(P,w) = \{w\}$). 
    Since we generally assume Uniqueness in what follows, we may refer to either condition as the ``Centering'' condition. 
\end{remark}


As discussed in Section~\ref{Section:Introduction}, we are interested in a semantics in which selections are made with respect to \emph{propositions} rather than formulas (see further discussion below). 
Apart from this, our Definition~\ref{Definition:Set Selection Frame} is the same as D3.4 in \cite[p.~27]{StalnakerThomason1970}, and our conditions 1--4 in Definition~\ref{Definition:Stalnakerian Function} pick up the same work as the correspondingly numbered conditions in D3.7 in \cite[p.~29]{StalnakerThomason1970}. 

Further, note that Uniformity and Uniqueness already imply a weak version of LA:
\footnote{
    Suppose that $f(P,w) = \varnothing$. 
    If $f(Q,w)\subseteq P$, then $f(P,w)=f(Q,w)=\varnothing$ by Uniformity, so $P\cap f(Q,w)=\varnothing$. 
    Alternatively, if $f(Q,w) \nsubseteq P$, then $f(Q,w)=\{w'\}$ for some $w'\not\in P$ by Uniqueness, whence $P \cap f(Q,w)=\varnothing$, as desired.
    } 
    \begin{enumerate}
        \item[3$^-$.] If $f(P,w) = \varnothing$, then $P \cap f(Q,w) = \varnothing$ \hfill (WLA)
    \end{enumerate}
It is clear that WLA is implied by LA, since by Definition~\ref{Definition:Set Selection Frame}, $P \cap f(Q,w) \subseteq P \cap R(w)$. 
On the other hand, WLA does not imply LA. This brings us to the following observation:

\begin{remark}
\label{Remark:WeakStalnakerian}
    Every Stalnakerian selection frame is weakly Stalnakerian, but not vice versa. 
    Consider the globally constant domain selection frame $\mathcal{F} = \tup{\{1,2\}, R, f, \{a\}}$, in which $R$ is universal and $f$ is defined as follows:
    \begin{align*}
        f(P,w) & = {
        \begin{cases}
            \{w\} & \text{if } w \in P \\
            \varnothing & \text{otherwise}
        \end{cases}
        }
    \end{align*}
    As is easily verified, $\mathcal{F}$ is weakly Stalnakerian but not Stalnakerian since, for example, $f(\{2\},1)=\varnothing$ but $\{2\}\cap R(1)\neq\varnothing$.
\end{remark}

While the significance of the distinction between Stalnakerian and weakly Stalnakerian frames may not be immediately apparent to the reader, as we will see, the logic identified by \cite{StalnakerThomason1970} turns out to be valid over precisely the \textit{weakly} Stalnakerian frames (Proposition~\ref{Proposition:FrameCorrespondence}).
Since the Stalnakerian frames are a proper subclass of the weakly Stalnakerian frames, in proving incompleteness with respect to the latter we will (a fortiori) prove it with respect to the former as well.


\begin{definition}
    \label{Definition:SelectionModel}
    A \emph{selection model} is a sextuple $\mathcal{M}=\tup{W,R,f,D,d,I}$, where $\tup{W,R,f,D,d}$ is a selection frame (Definition~\ref{Definition:Set Selection Frame}) and for every $w\in W$ and $n$-ary predicate $P^n$, $I(P^n,w)\subseteq D^n$. 
    We say $\mathcal{M}$ is \emph{based on} the frame $\tup{W,R,f,D,d}$. 
\end{definition}

\begin{definition}
    \label{Definition:VariableAssignment}
    A \emph{variable assignment} in a selection model $\tup{W,R,f,D,d,I}$ is a function $g:\mathcal{V}\to{D}$, where $\mathcal{V}$ is the set of variables. 
    A variable assignment $g'$ is an \emph{$x$-variant} of $g$ (in symbols, $g'\sim_{x}g$) if $g(y)=g'(y)$ for all variables $y$ except possibly $x$. 
    For $a\in D$, we write $g^x_a$ for the $x$-variant of $g$ such that $g^x_a(x) = a$. 
\end{definition}

\begin{definition}
\label{Definition:Satisfaction}
    For a given selection model $\mathcal{M} = \tup{W,R,f,D,d,I}$, world $w\in W$, and variable assignment $g$, we define $\Vdash$ as follows (we write $[\phi]^{g}$ for $\{w\in W \mid \mathcal{M},w,g \Vdash \phi\}$):    
    \begin{itemize}
        \item[At.] $\mathcal{M},w,g \Vdash P^{n}(x_1,\dots,x_n)$ iff $\tup{g(x_1),\dots,g(x_n)}\in I(P^{n},w)$;
        \item[$E$.] $\mathcal{M},w,g \Vdash E(x)$ iff $g(x) \in d(w)$; 
        \item[$=$.] $\mathcal{M},w,g \Vdash x = y$ iff $g(x) = g(y)$; 
        \item[$\neg$.] $\mathcal{M},w,g \Vdash \neg\phi$ iff $\mathcal{M},w,g\not\Vdash \phi$;
        \item[$\then$.] $\mathcal{M},w,g \Vdash \phi\then\psi$ iff $\mathcal{M},w,g\not\Vdash \phi$ or $\mathcal{M},w,g \Vdash \psi$;
        \item[$\forall$.] $\mathcal{M},w,g \Vdash \forall{x}\phi$ iff for all $a \in d(w)$, $\mathcal{M},w,g^x_a \Vdash \phi$;
        \item[$\would$.] $\mathcal{M},w,g \Vdash \phi \would \psi$ iff $f([\phi]^{g},w)\subseteq[\psi]^{g}$.
    \end{itemize}
    Where $\Gamma$ is a set of formulas, we write $\mathcal{M},w,g \Vdash \Gamma$ to mean that $\mathcal{M},w,g \Vdash \gamma$ for every $\gamma\in\Gamma$. 
    We write $\mathcal{M},w \Vdash \phi$ to mean that $\mathcal{M},w,g \Vdash \phi$ for every variable assignment $g$. 
\end{definition}


Note that in any model over a \emph{Stalnakerian} selection frame, the defined modal connectives $\nec$ and $\pos$ have their standard Kripke truth conditions.
\footnote{
    For illustration, consider the case of $\nec$, and suppose $\mathcal{M},w,g \Vdash \nec\phi$, that is, $\mathcal{M},w,g \Vdash \neg\phi \would \bot$. 
    Then $f([\neg\phi]^{g},w)=\varnothing$, so by LA, $[\neg\phi]^{g}\cap R(w)=\varnothing$, whence $R(w)\subseteq [\phi]^{g}$. 
    Conversely, if $R(w)\subseteq[\phi]^{g}$, then since $f([\neg\phi]^{g},w)\subseteq R(w) \subseteq[\phi]^{g}$ (Definition~\ref{Definition:Set Selection Frame}) and $f([\neg\phi]^{g},w)\subseteq [\neg\phi]^g$ (Success), $f([\neg\phi]^{g},w)=\varnothing\subseteq [\bot]^{g}$, as desired.
    } 
\begin{itemize}
    \item[$\nec$.] $\mathcal{M},w,g \Vdash \nec\phi$ iff $R(w)\subseteq[\phi]^{g}$;
    \item[$\pos$.] $\mathcal{M},w,g \Vdash \pos\phi$ iff $R(w)\cap[\phi]^{g}\neq\varnothing$.
\end{itemize}
This does not generally hold for models over \emph{weakly} Stalnakerian frames, however.
\footnote{
    Define a model $\mathcal{M}$ over the frame from Remark~\ref{Remark:WeakStalnakerian} by putting $I(P,1)=\{a\}$ and $I(P,2)=\varnothing$ (otherwise arbitrary). For any $g$, $\mathcal{M},1,g \Vdash \nec{P(x)}$, but $R(1)\not\subseteq[P(x)]^{g}$.
    } 
For this reason, we exercise caution in using the defined modal connectives hereafter, and generally expand them to conditional formulations.

\begin{definition}
\label{Definition:SelectionValidity}
    We say that $\Gamma$ is \emph{valid in the selection model} $\mathcal{M} = \tup{W,R,f,D,d,I}$ (in symbols, $\mathcal{M} \Vdash \Gamma$) iff for every world $w\in W$ and variable assignment $g$, $\mathcal{M},w,g \Vdash \Gamma$; that $\Gamma$ is \emph{valid in the selection frame} $\mathcal{F} = \tup{W,R,f,D,d}$ (in symbols, $\mathcal{F} \Vdash \Gamma$) iff $\mathcal{M} \Vdash \Gamma$ for every model $\mathcal{M} = \tup{\mathcal{F},I}$; and that $\Gamma$ is \emph{valid in the class of selection frames} $\mathcal{C}$ (in symbols, $\mathcal{C} \Vdash \Gamma$) iff for every $\mathcal{F}\in\mathcal{C}$, $\mathcal{F} \Vdash \Gamma$. 
    
    By $\mathsf{L}(\mathcal{C})$ we mean the set of formulas valid in the class of frames $\mathcal{C}$. Regarding a logic $\mathsf{L}$ as a set of formulas, to say that $\mathsf{L}$ is \emph{selection frame incomplete} is just to say that $\mathsf{L}\neq\mathsf{L}(\mathcal{C})$ for any class of selection frames $\mathcal{C}$. 
\end{definition}


Before concluding this section, we briefly return to two previously mentioned points. First, to point out a logical difference between globally and locally constant domains, observe that while $\forall x\phi \then \phi[y/x]$ is valid over the class of globally constant domain selection frames, it is not valid over the class of locally constant domain selection frames, since $y$ may pick out an object that is not in the local domain of the reference world or any (transitively) accessible world. This will be one source of subtlety when we turn to axiomatizations in Section~\ref{Section:Logics}.

Second, it will be instructive to quickly summarize the significant differences between our semantics and that of \cite{StalnakerThomason1970}.\footnote{In this presentation, we ignore insignificant differences (e.g., their use of the technical device $\lambda$).} For a given (global) domain $D$, let $\mathcal{VA}_D$ be the set of variable assignments $g\colon \mathcal{V} \rightarrow D$. 

\begin{definition}
\label{Definition:Quasi Selection Frame}
    A \emph{quasi-selection frame} is a quintuple $\tup{W,R,f,D,d}$ where $W$, $R$, $D$, and $d$ are as in Definition~\ref{Definition:Set Selection Frame} and $f\colon \mathcal{L}_{=} \times W \times \mathcal{VA}_D \rightarrow \mathcal{P}(W)$ is a \emph{quasi-selection function} mapping formulas, worlds, and variable assignments to sets of worlds; as before, it is required that $f(\phi,w,g) \subseteq R(w)$. 
\end{definition}

The notions of a (weakly) Stalnakerian quasi-selection frame, etc., are defined, \textit{mutatis mutandis}, as before. We define $\Vdash$ as in Definition~\ref{Definition:Satisfaction}, except for the $\would$ clause:
\begin{itemize}
    \item[$\would$.] $\mathcal{M},w,g \Vdash \phi \would \psi$ iff $f(\phi,w,g) \subseteq [\psi]^g$.
\end{itemize}
Validity with respect to quasi-selection models, etc., is defined, \textit{mutatis mutandis}, as in Definition~\ref{Definition:SelectionValidity}. 
We forego any extensive critique of the quasi-selection function semantics here, but we think it is clear that having selection functions take both formulas and variable assignments as arguments gives it a decidedly syntactic flavor, contrary to the spirit of Stalnaker's project. 


\subsection{Ordering Semantics}
\label{Subsection:Ordering Semantics}

In this section, we briefly present an alternative \emph{ordering} (or \emph{comparative similarity}) semantics for conditional logic. This semantics has been presented in two somewhat different forms in \cite{Lewis1971} and \cite[\S2.3]{Lewis1973}.
\footnote{
    In both cases, the semantics is only presented for propositional conditional logic. 
    As far as we are aware, there is no discussion in the literature of this semantics in connection with quantified conditional logic.
    } 
Modulo minor details, we more closely follow the original presentation below.

\begin{definition}
\label{Definition:Ordering frame}
    An \emph{ordering frame} is a quintuple $\tup{W,R,\preceq,D,d}$, where $W$, $R$, $D$, and $d$ are as in Definition~\ref{Definition:Set Selection Frame}, and for every $w\in W$, $\preceq_{w}\subseteq R(w)\times R(w)$. 
    An \emph{ordering model} is a sextuple $\tup{W,R,\preceq,D,d,I}$ where $\tup{W,R,\preceq,D,d}$ is an ordering frame and $I$ is an interpretation in the sense of Definition~\ref{Definition:SelectionModel}. 
\end{definition}

\begin{definition}
    \label{Definition:Stalnakerian Order}
    An ordering frame $\tup{W,R,\preceq,D,d}$ is \emph{Stalnakerian} just in case it satisfies these conditions for all $S\subseteq W$, $w \in W$:
    \begin{enumerate}
        \item $w\in R(w)$ \hfill (Reflexivity)
        \item $\forall x, y, z \in R(w): x\preceq_{w} y$ and $y\preceq_{w}z \Rightarrow x\preceq_{w}z$ \hfill (Transitivity)
        \item $\forall x, y \in R(w): x\preceq_{w} y$ or $y\preceq_{w}x$ \hfill (Strongly Connected)
        \item $x\in R(w) \Rightarrow w\preceq_{w}x$ \hfill (Weak Centering)
        \item $x\preceq_{w}w \Rightarrow x=w$ \hfill (Strong Centering)
        \item $S\cap R(w)= \varnothing$ or $\exists{x}\in S\cap R(w): \forall{y}\in S(y\preceq_{w}x \Rightarrow x=y)$ \hfill (SLA)
    \end{enumerate}
    An ordering frame is \emph{Lewisian} if it satisfies all of the conditions above except, possibly, SLA. 
\end{definition}


\cite[p.~77]{Lewis1971} restricts SLA for ordering semantics so as to only apply to expressible propositions, whereas we have presented it in full generality. 
The frames which satisfy all the conditions except possibly SLA characterize Lewis's preferred conditional logic, $\mathsf{C1}$ (alias: $\mathsf{VC}$), on which basis we call them ``Lewisian'' (see \cite[p.~83]{Lewis1971}).

The definitions of \emph{ordering model} and \emph{variable assignment} are, mutatis mutandis, the same as Definition~\ref{Definition:SelectionModel} and \ref{Definition:VariableAssignment}, respectively. 
For a given ordering model $\mathcal{M} = \tup{W,R,\preceq,D,d,I}$, world $w\in W$, and variable assignment $g$, $\Vdash$ is defined as in Section~\ref{Subsection:Set Selection Function Semantics} except:
\begin{itemize}
    \item[$\would$.] $\mathcal{M},w,g \Vdash \phi \would \psi$ iff $[\phi]^{g}\cap R(w)=\varnothing$ or $\exists{x}\in [\phi]^{g}\cap R(w): \forall{y}\in [\phi]^{g}(y\preceq_{w}x \Rightarrow y \in [\psi]^{g})$.
\end{itemize}
Validity with respect to models, frames, etc., is defined essentially as before (Definition~\ref{Definition:SelectionValidity}), and the notation is extended in the obvious way. 

The truth condition for $\would$ reduces to something like the selection semantics given SLA (i.e., if $\preceq_w$ is well-founded on $R(w)$). 
That is, defining $\min_{\preceq_w}(S):=\{x\in S\cap{R(w)}:\forall{y\in{S\cap{R(w)}}}(x\preceq_{w}y)\}$, we have:
\begin{itemize}
    \item[$\would$.] Given SLA: $\mathcal{M},w,g \Vdash \phi \would \psi$ iff $\min_{\preceq_w}([\phi]^g) \subseteq [\psi]^g$.
\end{itemize}
Thus, Stalnakerian selection models and Stalnakerian ordering models are equivalent. By defining $f(P,w) \coloneqq \min_{\preceq_w}(P)$, one can show that every Stalnakerian ordering model is equivalent to a Stalnakerian selection model. Conversely, by defining $\preceq_w \ \coloneqq \{\tup{v,u} \mid v \in f(\{v,u\},w)\}$, one can show that every Stalnakerian selection model is equivalent to a Stalnakerian ordering model.

On the other hand, Lewisian ordering models are more general than even weakly Stalnakerian models, even taking $f(P,w) \coloneqq \min_{\preceq_w}(P)$. 
For one, Lewisian ordering models do not require $\preceq_w$ to be a linear order.\footnote{In fact, in Lewisian ordering frames, $\preceq_w$ need not even be a partial order. However, as SLA implies antisymmetry ($x\preceq_w y$ and $y\preceq_w x$ imply $x=y$), Stalnakerian orders must be linear.} 
This is the analogue of the Uniqueness constraint on $f$, which corresponds to CEM, $(\phi \would \psi) \vee (\phi \would \neg\psi)$, in the selection semantics, and is not valid over the class of Lewisian ordering frames. 
Moreover, and perhaps more importantly for our purposes, Lewisian ordering models allow for violations of even WLA. 
In these cases, the constraint on selection functions that is typically violated for $\min_{\preceq_w}$ is Uniformity, which corresponds to CSO, $((\phi \would \psi) \wedge (\psi \would \phi) \wedge (\phi \would \chi)) \then (\psi \would \chi)$, in the selection function semantics (for further discussion, see Section~\ref{Section:Proof of Incompleteness}). 
Despite this, CSO is still valid over Lewisian ordering frames, thanks to the way the truth conditions for $\would$ are defined. 

\section{Quantified Conditional Logics}
\label{Section:Logics}

The aim of this paper is to establish incompleteness for a range of quantified conditional logics. 
Here, we present the original logic from \cite{StalnakerThomason1970}, as well as several variants. 

\begin{definition}[Abbreviations]
    We write $\phi[y_1,\dots,y_n/x_1,\dots,x_n]$ for the result of simultaneously substituting all free occurrences of $x_1,\dots,x_n$ in $\phi$ with $y_1,\dots,y_n$, replacing bound occurrences of $y_i$ with distinct variables as needed. 
    
    Where $\alpha_1,\dots,\alpha_n$ are formulas, we write $\vec{\alpha}$ for the sequence $\tup{\alpha_1,\dots,\alpha_n}$. 
    We write $\vec{\alpha} \would \phi$ for the right-nested conditional with antecedents $\alpha_1,\dots,\alpha_n$:
    \begin{align*}
        \vec{\alpha} \would \phi & \coloneqq (\alpha_1 \would (\alpha_2 \would (\cdots \would (\alpha_n \would \phi)\cdots)))
    \end{align*}
    In the case where $\vec{\alpha} = \tup{}$, we define $\vec{\alpha} \would \phi$ to just be $\phi$. 
\end{definition}

\cite[pp.\,30--31]{StalnakerThomason1970} study the following formal system in the language $\mathcal{L}_=$, which we label $\mathsf{QST}$ (what Stalnaker and Thomason label $\mathsf{CQ}$).
\footnote{
    Technically, Stalnaker and Thomason work in a language with both variables and (non-rigid) constants. 
    Hence, the presentation of this logic invokes $t$ and $s$ for arbitrary terms. 
    The presence of non-rigid constants does not substantively affect any of the results to follow, so we work with a system without such constants for simplicity. 
    } 

\begin{definition}
\label{def:ST}
    The logic $\mathsf{QST}$ in the language $\mathcal{L}_=$ is defined by the following axioms and rules:\footnote{\cite[p.~25]{StalnakerThomason1970} define $\nec\phi:=\neg\phi\would\phi$, which is equivalent to our abbreviation (see Definition~\ref{def:lang}) in any normal conditional logic containing the schema $\phi\would\phi$ (a fortiori, in \textsf{QST}).} 
    \begin{enumerate}
        \item $\phi$, where $\phi$ is a substitution instance of a propositional tautology
        \item $\nec(\phi \then \psi) \then (\nec\phi \then \nec\psi)$
        \item $\nec(\phi \then \psi) \then (\phi \would \psi)$
        \item $\pos\phi \then (\phi \would \psi \then \neg(\phi \would \neg\psi))$
        \item $(\phi \would (\psi \vee \chi)) \then ((\phi \would \psi) \vee (\phi \would \chi))$
        \item $(\phi \would \psi) \then (\phi \then \psi)$
        \item $((\phi \would \psi) \wedge (\psi \would \phi)) \then ((\phi \would \chi) \then (\psi \would \chi))$
        \item $\forall x\phi \then (\exists x\nec x=t \then \phi[t/x])$
        \item $\forall x (\exists y \nec y=x \then \phi) \then \forall x \phi$
        \item $s=s$
        \item $s=t\then (\phi \then \phi')$ where $\phi'$ is the result of replacing zero or more occurrences of $t$ with $x$ and where no occurrence of $t$ in $\phi$ that is replaced by $s$ in $\phi'$ occurs in the scope of a modal
        \item $\pos x=y \then \nec x=y$
        \item if $\vdash \phi$ and $\vdash \phi \then \psi$, then $\vdash \psi$
        \item if $\vdash \phi$ then $\vdash \psi \would \phi$
        \item if $\vdash \psi \then \phi$ and $\psi$ has no free occurrences of $x$, then $\vdash \psi \then \forall x\phi$
        \item if $\vdash \vec{\alpha} \would \phi$ and $x$ does not occur free in any $\alpha_i$, then $\vdash \vec{\alpha} \would \forall x\phi$
        \item if $\vdash \vec{\alpha} \would (\phi \would t \neq x)$ and $x$ does not occur free in any $\alpha_i$ or in $\phi$, then $\vdash \vec{\alpha} \would \neg\phi$.
        \footnote{
            In a framework without non-rigid constants, rule 17 is admissible given the other axioms and rules. 
            This effectively requires showing the following (by induction on the length of proofs): if $\vdash \phi[y/x]$, then $\vdash \phi[z/x]$. 
            Thus, if $\vdash \vec{\alpha} \would (\phi \would y \neq x)$ where $x$ doesn't occur free in $\vec{\alpha}$ or $\phi$, then $\vdash \vec{\alpha} \would (\phi \would y \neq y)$, which implies $\vdash \vec{\alpha} \would \neg\phi$. 
            } 
    \end{enumerate}
\end{definition}

\cite{StalnakerThomason1970} prove that $\mathsf{QST}$ is sound and complete with respect to certain variable domain selection frames in which selection functions take formulas as arguments rather than propositions (see Definition~\ref{Definition:Quasi Selection Frame}).

\begin{theorem}[\cite{StalnakerThomason1970}]
\label{thm:QST-comp}
    $\mathsf{QST}$ is sound and complete in $\mathcal{L}_=$ for the class of Stalnakerian quasi-selection models. 
\end{theorem}

For simplicity but also to show the robustness of our result, in what follows, we will largely focus on other systems related to $\mathsf{QST}$. 
These systems will be easier to generalize to frameworks beyond the one studied by Stalnaker and Thomason (e.g., to constant domains, to languages without identity, etc.).  

The simplest system, which we will spend the most time working with, is the system \textsf{QC2}. 
This system is effectively the result of combining the propositional conditional logic \textsf{C2} with the standard axioms and rules for classical first-order quantification. 
More precisely:

\begin{definition}
\label{def:QC2}
    The logic $\mathsf{QC2}$ in the language $\mathcal{L}$ is defined by the following axioms and rules:
    \begin{enumerate}
    \setcounter{enumi}{17}
        \item $\phi$, where $\phi$ is a substitution instance of a propositional tautology\label{ax:PL}
        \item $\phi \would \phi$\label{ax:id}
        \item $((\phi \would \psi) \wedge (\psi \would \phi) \wedge (\phi \would \chi)) \then (\psi \would \chi)$\label{ax:cso}
        \item $(\phi \would \psi) \then (\phi \then \psi)$\label{ax:cen}
        \item $(\phi \would \psi) \vee (\phi \would \neg\psi)$\label{ax:cem}
        \item $\forall x\phi \then \phi[y/x]$\label{ax:uelim-c}
        \item $\forall x(\phi \would \psi) \then (\phi \would \forall x \psi)$ where $x$ isn't free in $\phi$\label{ax:barcan-gc}
        \item if $\vdash \phi$ and $\vdash \phi \then \psi$, then $\vdash \psi$\label{rule:mp}
        \item if $\vdash (\psi_1 \wedge \cdots \wedge \psi_n) \then \chi$, then $\vdash ((\phi \would \psi_1) \wedge \cdots \wedge (\phi \would \psi_n)) \then (\phi \would \chi)$\label{rule:cnec}
        \item if $\vdash \psi \then \phi[y/x]$, where $y$ doesn't occur free in $\psi$ or $\forall x\phi$, then $\vdash \psi \then \forall x\phi$.\label{rule:uintro-c}
    \end{enumerate}
    The logic $\mathsf{QC2_=}$ in the language $\mathcal{L}_=$ is the result of extending $\mathsf{QC2}$ with the following axioms:
    \begin{enumerate}
    \setcounter{enumi}{27}
        \item $x=x$\label{ax:reflx}
        \item $x = y \then (\phi \equiv \phi[y/x])$\label{ax:subid}
        \item $x \neq y \then \nec x \neq y$\label{ax:necdec}
    \end{enumerate}
\end{definition}

Briefly, $\mathsf{QC2}$ is the globally constant domain analogue of $\mathsf{QST}$ without identity, while $\mathsf{QC2_=}$ is the globally constant domain analogue of $\mathsf{QST}$ with identity.
(We don't bother including ``$\mathsf{QC2_{\textit{E}}}$'', which would just add $Ex$ as an axiom to $\mathsf{QC2}$ formulated in $\mathcal{L}_{E}$.) 

One way to see that these logics are the globally constant domain analogues of $\mathsf{QST}$ is to note the following theorem, the proof of which we leave as an exercise to the reader.
\footnote{
    Observe that rule 16 of $\mathsf{QST}$ can be shown to be admissible using axiom~\ref{ax:barcan-gc} and rule~\ref{rule:uintro-c} of $\mathsf{QC2}$. 
    } 

\begin{theorem}
\label{thm:QC2-comp}
    $\mathsf{QC2}_{(=)}$ is sound and complete in $\mathcal{L}_{(=)}$ for the class of globally constant domain Stalnakerian quasi-selection models. 
\end{theorem}

\noindent In other words, $\mathsf{QC2}_{(=)}$ is sound and complete for the class of models where selection functions take formulas (and variable assignments) as inputs, rather than propositions. 
Thus, $\mathsf{QST}$ has the same status towards variable domains that $\mathsf{QC2}_=$ has towards globally constant domains. 

Our first incompleteness result will pertain to $\mathsf{QC2}$. 
However, we also prove corresponding incompleteness results for $\mathsf{QC2}_=$ as well as the variable domain analogues of $\mathsf{QC2}$ defined below.

\begin{definition}
\label{def:QC2-v}
    The logic $\mathsf{QC2^v_{\textit{E}}}$ in $\mathcal{L}_E$ is defined by replacing axiom~\ref{ax:uelim-c} and rule~\ref{rule:uintro-c} in $\mathsf{QC2}$ with the following:
    \begin{enumerate}[label={\arabic*$^{\mathsf{v}}$.},itemsep=1mm]
    \setcounter{enumi}{22}
        \item $(\forall x\phi \wedge E(y)) \then \phi[y/x]$\label{ax:uelim-v}
    \setcounter{enumi}{26}
        \item If $\vdash \psi \then (\vec{\alpha} \would (E(y) \supset \phi[y/x]))$, where $y$ does not occur free in $\alpha_1,\dots,\alpha_n$, $\psi$, or $\forall x\phi$, then $\vdash \psi \then (\vec{\alpha} \would \forall x\phi)$.\label{rule:uintro-v}
    \end{enumerate}
    The logic $\mathsf{QC2^c_{\textit{E}}}$ in $\mathcal{L}_E$ is the result of extending $\mathsf{QC2^v_{\textit{E}}}$ with the following axioms:
    \begin{enumerate}[label={\arabic*$^{\mathsf{c}}$.}]
    \setcounter{enumi}{30}
        \item $E(x) \then \nec{E(x)}$\label{ax:nondec}
        \item $\neg E(x) \then \nec\neg E(x)$\label{ax:noninc}
    \end{enumerate}
    The logics $\mathsf{QC2^v_=}$ and  $\mathsf{QC2^c_=}$ are defined as the result of adding axioms~\ref{ax:reflx}--\ref{ax:necdec} to $\mathsf{QC2^v_{\textit{E}}}$ and $\mathsf{QC2^c_{\textit{E}}}$ respectively. 
\end{definition}

Effectively, $\mathsf{QC2^v_{\textit{E}/=}}$ are the variable domain analogues of $\mathsf{QC2_{(=)}}$ while $\mathsf{QC2^c_{\textit{E}/=}}$ are the \emph{locally} constant domain analogues.
\footnote{
    Note that we have not axiomatized a logic ``$\mathsf{QC2^v}$'' for variable domain frames in the language $\mathcal{L}$ \emph{without} $E$ or $=$, or a logic ``$\mathsf{QC2^c}$'' for locally constant domain frames in $\mathcal{L}$. 
    We were unable to find natural logics for which one could prove completeness with respect to the corresponding classes of quasi-selection models. 
    This may seem somewhat surprising, but upon reflection, it is to be expected since proving completeness for quantified modal logics over variable domains in languages without $E$ or $=$ is notably more difficult than proving completeness in languages with these devices (cf.~\cite[pp.\ 304ff.]{HughesCresswell1996}). 
    } 
Again, we can see this by noting the following completeness theorem:

\begin{theorem}
\label{thm:QC2v-comp}
    $\mathsf{QC2^v_{\textit{E}/=}}$ is sound and complete in $\mathcal{L}_{E/=}$ for the class of variable domain Stalnakerian quasi-selection models. 
    $\mathsf{QC2^c_{\textit{E}/=}}$ is sound and complete in $\mathcal{L}_{E/=}$ for the class of locally constant domain Stalnakerian quasi-selection models. 
\end{theorem}

It follows from Theorem~\ref{thm:QC2v-comp} that $\mathsf{QC2^v_=}$ is provably equivalent to the logic $\mathsf{QST}$. 
Indeed, one can verify this directly (i.e., axiomatically) by showing that they contain each other. 
Again, we leave this as an exercise to the reader. 
An important step in this exercise is to recognize that the following, known as the MOD axiom (see, e.g., \cite[p.~160]{Nute1980}), is a theorem of $\mathsf{QC2}$ and each of its variations (use axioms~\ref{ax:cso} and axiom~\ref{ax:cem}):
\begin{align*}
    (\neg\phi \would \bot) \then (\psi \would \phi) \ \ (\text{i.e., } \nec\phi \then (\psi \would \phi))
\end{align*}
This axiom corresponds to the frame condition WLA from Section~\ref{Subsection:Set Selection Function Semantics}, i.e., for any selection frame $\mathcal{F}$, $\mathcal{F} \Vdash \nec\phi \then (\psi \would \phi)$ just in case $\mathcal{F}$ satisfies WLA.


Certain classes of weakly Stalnakerian frames are exactly characterized by $\mathsf{QC2}$ or its variants in the following sense:

\begin{proposition}
\label{Proposition:FrameCorrespondence}
    Let $\mathcal{F} = \tup{W,R,f,D,d}$ be a selection frame. 
    Then $\mathsf{QC2_{(=)}}$ is valid over $\mathcal{F}$ (i.e., $\mathcal{F} \Vdash \mathsf{QC2_{(=)}}$) just in case $\mathcal{F}$ is weakly Stalnakerian and globally constant. 
    Similarly, $\mathsf{QC2^v_{\ex/=}}$ is valid over $\mathcal{F}$ just in case $\mathcal{F}$ is weakly Stalnakerian, and $\mathsf{QC2^c_{\ex/=}}$ is valid over $\mathcal{F}$ just in case $\mathcal{F}$ is weakly Stalnakerian and locally constant. 
\end{proposition}

\begin{proof}
    We consider a few representative cases. 
    We'll only present the proposition for $\mathsf{QC2}$, leaving the corresponding results for the other logics to the reader. 
    

    The right-to-left direction is routine---it essentially amounts to checking the soundness of the axioms and rules of $\mathsf{QC2}$ over globally constant weakly Stalnakerian frames. 
    We'll just illustrate the soundness of axiom~\ref{ax:barcan-gc}: $\forall x(\phi \would \psi) \then (\phi \would \forall x\psi)$ when $x$ doesn't occur free in $\phi$. 
    Suppose $\mathcal{F}$ is weakly Stalnakerian and globally constant. 
    Let $\mathcal{M}$ be based on $\mathcal{F}$ where $\mathcal{M},w,g \Vdash \forall x(\phi \would \psi)$. 
    Thus, for all $a \in D$, $f([\phi]^{g^x_a},w) \subseteq [\psi]^{g^x_a}$. 
    Since $x$ doesn't occur free in $\phi$, $[\phi]^{g^x_a} = [\phi]^g$. 
    Hence, $f([\phi]^g,w) \subseteq \bigcap_{a \in D} [\psi]^{g^x_a} = [\forall x\psi]^g$, that is, $\mathcal{M},w,g \Vdash \phi \would \forall x\psi$. 

    For the left-to-right direction, we'll just show that if $\mathcal{F} \Vdash \mathsf{QC2}$, then $\mathcal{F}$ is globally constant and obeys Uniformity and Uniqueness (Definitions~\ref{Definition:Domains} and \ref{Definition:Stalnakerian Function}). 
    \begin{itemize}
        \item Globally constant: For any $w\in W$, $d(w)\subseteq{D}$, so suppose $a\in D$ and consider a model $\mathcal{M}$ over $\mathcal{F}$ such that $I(F,w) = d(w)$ and a variable assignment $g$ such that $g(y) = a$ (otherwise arbitrary). Then $\mathcal{M},w,g \Vdash \forall{x}F(x)$, so by the hypothesis that $\mathcal{F}$ validates axiom~\ref{ax:uelim-c}, we have that $\mathcal{M},w,g \Vdash F(y)$, that is, $g(y) = a \in I(F,w) = d(w)$. 

        \item Uniformity: For $P, Q \subseteq W$, suppose that $f(P,w) \subseteq Q$ and $f(Q,w) \subseteq P$. 
        Consider any model $\mathcal{M}$ over $\mathcal{F}$ such that $I(A,w') = D$ if $w' \in P$ (otherwise empty); $I(B,w') = D$ if $w' \in Q$ (otherwise empty); and, finally, $I(C,w') = D$ if $w'\in f(P,w)$ (otherwise empty). 
        Then, for any variable assignment $g$, $\mathcal{M},w,g \Vdash (A(x) \would B(x)) \wedge (B(x) \would A(x)) \wedge (A(x) \would C(x))$. 
        Since $\mathcal{F}$ validates axiom~\ref{ax:cso}, $\mathcal{M},w,g \Vdash B(x) \would C(x)$, that is, $f(Q,w) \subseteq f(P,w)$. 
        Symmetric reasoning establishes that $f(P,w) \subseteq f(Q,w)$, hence that $f(P,w) = f(Q,w)$, as desired.
        

        \item Uniqueness: If Uniqueness failed, then for some $P\subseteq{W}$ and $u, v, w \in W$, $u\neq{v}$ and $u, v \in f(P,w)$. Define a model $\mathcal{M}$ over $\mathcal{F}$ by putting $I(A,w')=D$ if $w'\in P$ and $I(B,w')=D$ if $w'=u$; otherwise, empty. Then for any variable assignment $g$, $v\in f([Ax]^{g},w)\not\subseteq\{u\}=[Bx]^{g}$ and $u\in f([Ax]^{g},w)\not\subseteq [\neg{Bx}]^{g} = W\setminus\{u\}$. Therefore, axiom~\ref{ax:cem} is not valid on $\mathcal{F}$, contradicting our assumption.
    \end{itemize}
\end{proof}

One might think that completeness follows immediately from Proposition~\ref{Proposition:FrameCorrespondence}. 
But this is not the case. 
While Proposition~\ref{Proposition:FrameCorrespondence} shows that $\mathsf{QC2}$ and its variants are valid on exactly weakly Stalnakerian frames, they are not the logics of any such classes of selection frames. 
This is because this class of frames validates a formula that is not provable in $\mathsf{QC2}$ or its variants. 
We turn to establishing this result in the next section. 

\section{Incompleteness for $\mathsf{QC2}$}
\label{Section:Proof of Incompleteness}

In this section we prove the simplest instance of our main theorem, the frame incompleteness of $\mathsf{QC2}$. 
This will serve as the basis for proving other incompleteness results, including the frame incompleteness of the Stalnaker-Thomason logic $\mathsf{QST}$. 

In proving this result, we make use of the following formula (cf.~\cite[p.~182]{Montagna1984provability}):
\begin{align*}
    \mathsf{DS} \coloneqq \exists x\top \wedge \forall x \pos F(x) \wedge \forall x \exists y \Big( \big( F(x) \vee F(y) \big) > \neg F(x)\Big).
\end{align*}
$\mathsf{DS}$ is so-called because it forces the existence of a ``descending sequence'' of worlds. 
In this section we will see that (i) $\neg \mathsf{DS}$ is valid in all weakly Stalnakerian frames yet (ii) $\mathsf{QC2}\nvdash \neg \mathsf{DS}$. From these it follows that $\mathsf{QC2}$ is not complete with respect to the class of weakly Stalnakerian frames. 
Proposition \ref{Proposition:FrameCorrespondence} then entails that $\mathsf{QC2}$ is not the logic of any class of selection frames. 

\begin{remark}
We write $\vec{x}$ for $x_1,\dots,x_n$ and $g(\vec{x})$ for $g(x_1),\dots,g(x_n)$. 
Where $\phi(\vec{x})$ only contains $\vec{x}$ free and $\vec{a} \in D$, we write $\mathcal{M},w \Vdash \phi(\vec{a})$ for $\mathcal{M},w,g^{\vec{x}}_{\vec{a}} \Vdash \phi(\vec{x})$. 
We likewise write $[\phi(\vec{a})]$ for $[\phi(\vec{x})]^{g^{\vec{x}}_{\vec{a}}}$. 
This notation is well-defined since (i) variables are rigid, i.e., $g$ is a function from variables to $D$, and (ii) when $g(\vec{x}) = g'(\vec{x})$, $\mathcal{M},w,g \Vdash \phi(\vec{x})$ iff $\mathcal{M},w,g' \Vdash \phi(\vec{x})$.  
\end{remark}

\begin{proposition}
\label{prop:DS}
    No weakly Stalnakerian frame satisfies $\mathsf{DS}$.
\end{proposition}
\begin{proof}
    Suppose $\mathcal{M},w \Vdash \mathsf{DS}$. 
    Let $Q_a = \{v \mid \mathcal{M},v \Vdash F(a)\}$ and let $P = \bigcup_{a \in d(w)} Q_a$. 
    By Uniqueness, $|f(P,w)| \leq 1$. 
    Suppose $f(P,w) = \varnothing$. 
    Then by WLA, $f(Q_a,w) \cap P = \varnothing$ for all $a \in D$. 
    But $f(Q_a,w) \subseteq Q_a \subseteq P$, so $f(Q_a,w) = \varnothing$, contrary to the fact that $\mathcal{M},w \Vdash \forall x \pos F(x)$ and $d(w) \neq \varnothing$ (since $\mathcal{M},w \Vdash \exists x \top$). 
    Suppose instead $f(P,w) = \{v\}$. 
    By definition of $P$, there is an $a \in d(w)$ such that $\mathcal{M},v \Vdash F(a)$. 
    Since $\mathcal{M},w \Vdash \exists y((F(a) \vee F(y)) \would \neg F(a))$, let $b \in d(w)$ be a witness. 
    So $\mathcal{M},w \Vdash (F(a) \vee F(b)) \would \neg F(a)$.
    By Success, $f(Q_a \cup Q_b,w) \subseteq Q_a \cup Q_b \subseteq P$. 
    But also $f(P,w) = \{v\} \subseteq Q_a \subseteq Q_a \cup Q_b$. 
    By Uniformity, $f(Q_a \cup Q_b,w) = f(P,w) = \{v\}$, contrary to the fact that $\mathcal{M},w \Vdash (F(a) \vee F(b)) \would \neg F(a)$. 
\end{proof}

\begin{remark}
    The proof of Proposition~\ref{prop:DS} is a simple extension of the argument in \cite{kocurek2025stalnaker} (cf.\ \cite[\S2.2]{dorr2024logic}) that, even in the propositional case, compactness fails over the class of Stalnakerian frames (e.g., $\{\pos A_i\}_{i \in \omega} \cup \{(A_i \vee A_{i+1}) \would \neg A_i\}_{i \in \omega}$ is finitely satifisfiable yet as a whole unsatisfiable). 
    Since the proof of Proposition~\ref{prop:DS} does not rely on LA (only WLA, which follows from Uniformity and Uniqueness), this shows that compactness fails over weakly Stalnakerian frames as well. 
    
    In fact, using a similar strategy to \cite{kocurek2025stalnaker}, the proof of Proposition~\ref{prop:DS} does not require Uniqueness and Uniformity: we only need Success, WLA, and Rational Monotonicity (cf.\ \cite[pp.~104--5; Prop II.59]{Veltman1985logics}): 
    \begin{align*}
        \text{if } P \subseteq Q \text{ and } f(Q,w) \cap P \neq \varnothing \text{, then } f(P,w) = f(Q,w) \cap P
    \end{align*}
    Taken together, the proof of Proposition~\ref{prop:DS} can be strengthened to show that no selection frame satisfying Success, WLA, and Rational Monotonicity satisfies $\mathsf{DS}$. 
    Furthermore, observe that the proof does not rely on any special constraints on local domains. 
\end{remark}

\begin{proposition}\label{counter-model}
$\mathsf{QC2} \nvdash \neg\mathsf{DS}$.
\end{proposition}

To prove Proposition \ref{counter-model}, we use the order semantics (Section~\ref{Subsection:Ordering Semantics}). In particular, we exhibit a globally constant order model of $\mathsf{QC2} + \neg \mathsf{DS}$.

\begin{definition}\label{order-model}
    We define a globally constant order model $\mathcal{K}=\tup{W, R, \preceq, D, I}$ satisfying the following conditions (for a visualization, see Figure~\ref{modelK}):
\begin{enumerate}
    \item $W = \mathbb{Z}^- \cup \{-\infty\}$, i.e., the set of negative integers plus an additional element $-\infty$ (where, for each integer $k$, $-\infty <k$);
    \item $R(-\infty) = W$ and $R(k) = \{k\}$;
    \item $D = \mathbb{Z}^- = d(w)$ for all $w \in W$;
    \item $x \preceq_{-\infty} y$ iff $x \leq y$;
    \item $\preceq_{k}\coloneqq\{\tup{k,k}\}$;
    \item $I(F,k) = \{n \in \mathbb{Z}^- \mid n \leq k\}$ for $k \neq -\infty$;
    \item $I(F,-\infty) = \varnothing$;
    \item $I(G,x) = \varnothing$ for every other predicate $G$.
\end{enumerate}
\end{definition}




\begin{figure}[ht!]
\begin{center}
\begin{tikzpicture}[scale=0.9, transform shape, node distance={15mm}, thin, main/.style = {draw}]
  \node[main, draw=none] (2) {$\dots$};
  \node[main] (3) [right of=2] {$-3$};
  \node[main] (4) [right of=3] {$-2$};
  \node[main] (5) [right of=4] {$-1$};
  \node[main, draw=none] (6) [above of=5] {\checkmark};
  \node[main, draw=none] (7) [above of=6] {\checkmark};
  \node[main, draw=none] (8) [above of=7] {\checkmark};
  \node[main, draw=none] (9) [above of=4] {\xmark};
  \node[main, draw=none] (10) [above of=9] {\checkmark};
  \node[main, draw=none] (11) [above of=10] {\checkmark};
  \node[main, draw=none] (12) [above of=11] {\vdots};
  \node[main, draw=none] (13) [above of=8] {\vdots};
    \node[main, draw=none] (14) [above of=3] {\xmark};
  \node[main, draw=none] (15) [above of=14] {\xmark};
  \node[main, draw=none] (16) [above of=15] {\checkmark};
  \node[main, draw=none] (17) [above of=16] {\vdots};
  \node[main, draw=none] (18) [right of=13] {\vdots};
  \node[main] (infty) [left of=2] {$-\infty$};
  \node[main, draw=none] (inf1) [above of=infty] {\xmark};
  \node[main, draw=none] (inf2) [above of=inf1] {\xmark};
  \node[main, draw=none] (inf3) [above of=inf2] {\xmark};
  \node[main, draw=none] (inf4) [above of=inf3] {\vdots};
  \node[main, draw=none] (19) [right of=6] {$-1$};
  \node[main, draw=none] (19) [right of=7] {$-2$};
  \node[main, draw=none] (19) [right of=8] {$-3$};
  \draw[->] (2) -- (3);
  \draw[->] (3) -- (4);
  \draw[->] (4) -- (5);
  \draw[->] (infty) -- (2);
\end{tikzpicture}
\end{center}
\caption{The intension of the predicate $F$ in structure $\mathcal{K}$. The $x$-axis represents worlds and the $y$-axis represents elements of the domain.}
\label{modelK}
\end{figure}
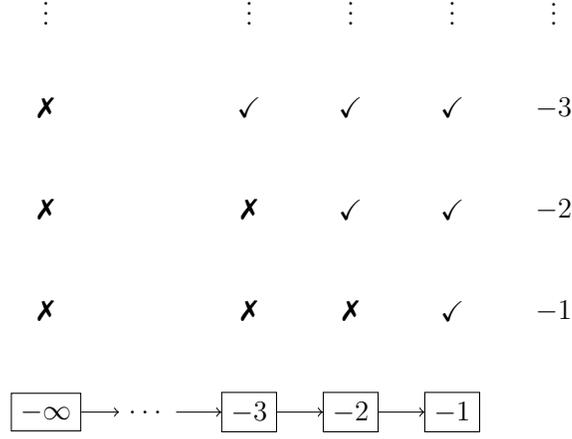

Observe that $\mathcal{K}$ is a Lewisian ordering model. 
However, it is not Stalnakerian, as the strong limit assumption (SLA) fails for this structure since $\preceq_{-\infty}$ is not a well-order: $\mathbb{Z}^{-}\cap{R(\neg\infty)}\neq\varnothing$ but for any $k\in \mathbb{Z}^{-}\cap{R(\neg\infty)}$, there is a $j\in \mathbb{Z}^{-}$ such that $j\prec_{-\infty}k$.
Despite this, we'll be able to show that $\mathsf{QC2}$ still holds at $-\infty$, since the definable sets of worlds in this model are rather constrained. 
Note, however, that we cannot establish this by simply defining $f(P,w) \coloneqq \min_{\preceq_w}(P)$ and then applying the selection function semantics. 
Indeed, by Propositions~\ref{Proposition:FrameCorrespondence} and \ref{prop:DS}, there is no selection frame that can simultaneously validate $\mathsf{QC2}$ while satisfying $\mathsf{DS}$.\footnote{Perhaps one can construct a selection model that satisfies both $\mathsf{DS}$ and $\mathsf{QC2}$ at a point without the underlying frame validating $\mathsf{QC2}$, but we have yet to find such a model.} 
In particular, if we take $f(P,w)$ to be the set of $\preceq_w$-minimal elements of $P$, we get a violation of Uniformity: $f(\mathbb{Z}^-,-\infty) = \varnothing \subseteq \{-1\}$ and $f(\{-1\},-\infty) = \{-1\} \subseteq \mathbb{Z}^-$, but $f(\mathbb{Z}^-,-\infty) \neq f(\{-1\},-\infty)$. 
So to show that $\mathsf{QC2}$ ``holds'' at $-\infty$, we will reason within the ordering semantics itself. 
This will suffice to establish that $\mathsf{QC2} \nvdash \neg\mathsf{DS}$: we only need a model of $\mathsf{QC2}$ in \emph{some} semantics that does not satisfy $\neg\mathsf{DS}$ to establish that $\neg\mathsf{DS}$ isn't a theorem of $\mathsf{QC2}$. 

Observe that if there is a $\preceq_w$-minimal $[\phi]^g$-world, then $\mathcal{K},w,g \Vdash \phi \would \psi$ iff $\min_{\preceq_w}([\phi]^g) \subseteq [\psi]^g$. 
In particular, if there is a $\preceq_w$-minimal $[\phi]^g$-world, then $\mathcal{K},w \Vdash \neg(\phi \would \bot)$.



\begin{claim}\label{descending sequences}
    $\mathcal{K},-\infty\Vdash \mathsf{DS}$.
\end{claim}

\begin{proof}
    For the first conjunct, this follows from the fact that the local domain of each $w \in W$ is non-empty. 
    
    For the second conjunct, observe that $\mathcal{K},k \Vdash F(k)$ (or, put differently, $\mathcal{K},g(x),g \Vdash F(x)$). 
    Moreover, $k$ is the closest $[F(k)]$-world to $-\infty$, which means $\mathcal{K},-\infty \Vdash \neg(F(k) \would \bot)$, i.e., $\mathcal{K},-\infty \Vdash \pos F(k)$.
    Thus, $\mathcal{K},-\infty \Vdash \forall x \pos F(x)$.

    For the third conjunct, let $k \in \mathbb{Z}^-$. 
    For the witness $y$, we choose $k-1$. 
    Observe the $\preceq_{-\infty}$-minimal element of $[F(k) \vee F(k-1)]$ is $k-1$. 
    Yet $\mathcal{K},k-1 \nVdash F(k)$. 
    Hence, $\mathcal{K},-\infty \Vdash ((F(k) \vee F(k-1)) \would \neg F(k)$.
\end{proof}

The next thing to check is that all of the theorems of $\mathsf{QC2}$ hold at $-\infty$. 
In fact, it is easier to prove the stronger claim that $\mathsf{QC2}$ is valid in $\mathcal{K}$, that is, $\mathcal{K},w,g \Vdash \mathsf{QC2}$ for \emph{all} worlds $w$ and variable assignments $g$. 
By far the hardest part of the proof is verifying that axiom~\ref{ax:cem}, $(\phi \would \psi) \vee (\phi \would \neg\psi)$, which we'll label $\mathsf{CEM}$, holds at $-\infty$. 
We isolate the following easy claims first.


\begin{claim}
\label{claim:reduceZ-}
    Where $\phi$ is a formula, let $\phi^-$ be the result of replacing every occurrence of $\would$ by $\then$ in $\phi$. 
    Then $[\phi]^g \cap \mathbb{Z}^- = [\phi^-]^g \cap \mathbb{Z}^-$. 
\end{claim}

\begin{proof}
    By induction on the complexity of $\phi$. 
    For the $\would$-case, since $R(k) = \{k\}$, it is easy to check that $\mathcal{K},k,g \Vdash \phi \would \psi$ iff $\mathcal{K},k,g \Vdash \phi \then \psi$. 
\end{proof}



\begin{claim}\label{without-cem}
    All of the axioms of $\mathsf{QC2}$ besides $\mathsf{CEM}$ are valid in $\mathcal{K}$.
    Moreover, the rules of $\mathsf{QC2}$ preserve validity in $\mathcal{K}$. 
\end{claim}

\begin{proof}
By Claim~\ref{claim:reduceZ-}, for any $k \in \mathbb{Z}^-$, $\mathcal{K},k,g\Vdash\alpha$ for any axiom $\alpha$ of $\mathsf{QC2}$ (including $\mathsf{CEM}$). 
Moreover, all of the axioms besides $\mathsf{CEM}$ are valid on the ordering semantics over Lewisian frames with globally constant domains (cf.~\cite{Lewis1971}). 
Hence, every axiom $\alpha$ besides $\mathsf{CEM}$ is valid in $\mathcal{K}$, and the rules preserve validity in $\mathcal{K}$. 

For illustration, we briefly verify that rule~\ref{rule:cnec} preserves validity. 
Suppose that $\mathcal{K} \Vdash (\psi_1 \wedge \cdots \wedge \psi_n) \then \chi$. 
Suppose $[\phi]^g \cap R(w) \neq \varnothing$ and for each $1\leq i\leq{n}$, there is a $k_i \in [\phi]^g \cap R(w)$ such that $\mathcal{K},j,g \Vdash \phi \then \psi_i$ for all $j \preceq_{w} k_i$. 
Then where $k = \min_{\preceq_w}(k_1,\dots,k_n)$, any $j \preceq_{w} k$ is such that $\mathcal{K},j,g \Vdash (\phi \then \psi_1) \wedge \cdots \wedge (\phi \then \psi_n) \wedge ((\psi_1 \wedge \cdots \wedge \psi_n) \then \chi$. 
Hence, for any such $j$, $\mathcal{K},j,g \Vdash \phi \then \chi$. 
So $\mathcal{K},w,g \Vdash \phi \would \chi$.
\end{proof}

This leaves checking that $\mathsf{CEM}$ is valid in $\mathcal{K}$. 
The reasoning above does not apply since even though $\preceq_w$ is linear, $\mathsf{CEM}$ is not guaranteed to hold without the well-foundedness of $\preceq_w$, that is, SLA. 
(For example, if there were an infinitely descending alternating sequence of $(\phi \wedge \psi)$-worlds and $(\phi \wedge \neg\psi)$-worlds, with no minimal $\phi$-world, $(\phi \would \psi) \vee (\phi \would \neg\psi)$ would not be satisfied.) 
Thus, a more sophisticated argument is required to show that the model nevertheless satisfies every instance of $\mathsf{CEM}$ expressible in the language. 

For $k \in \mathbb{Z}^-$, by Claim~\ref{claim:reduceZ-}, this amounts to checking that $\mathcal{K},k \Vdash (\phi \then \psi) \vee (\phi \then \neg\psi)$, which is straightforward. 
The more difficult case, of course, is showing $\mathcal{K},-\infty \Vdash \mathsf{CEM}$. 
In effect, the reason $\mathsf{CEM}$ holds even at $-\infty$ is that the sets expressible by a formula are very constrained. 

\begin{remark}\label{surjective}
    In what follows, we \emph{fix} formulas $\phi$ and $\psi$, with the aim of showing that $\mathcal{K},-\infty \Vdash (\phi \would \psi) \vee (\phi \would \neg\psi)$. 
    This means we need to show, for an \emph{arbitrary} variable assignment $g$, $\mathcal{K},-\infty, g \Vdash (\phi \would \psi) \vee (\phi \would \neg\psi)$. 
    Note, however, that if $g$ and $g'$ agree on all variables appearing in $\phi$ and $\psi$, then $[\phi \would \psi]^g = [\phi \would \psi]^{g'}$. 
    In particular, if $g^\star$ is a surjective variable assignment such that $g^\star(x) = g(x)$ for all $x$ appearing in $\phi$ and $\psi$, then $\mathcal{K},-\infty,g \Vdash (\phi \would \psi) \vee (\phi \would \neg\psi)$ iff $\mathcal{K},-\infty,g^\star \Vdash (\phi \would \psi) \vee (\phi \would \neg\psi)$. 
    Such a $g^\star$ is always guaranteed to exist since only finitely many variables appear in $\phi$ and $\psi$. 
    Hence, we restrict our attention---without loss of generality---to surjective variable assignments. 
    This allows us to write ``$\phi(k)$'' for $\phi(x)$ where $g^\star(x) = k$. 
\end{remark}

We now fix formulas $\phi$ and $\psi$ with the goal of showing that $\mathcal{K},-\infty \Vdash (\phi \would \psi) \vee (\phi \would \neg\psi)$. 
Note that this is trivially satisfied if $\mathcal{K},-\infty,g \Vdash \phi$, since $\mathcal{K}$ obeys Weak Centering. 
So we'll focus throughout on the case in which $\mathcal{K},-\infty,g \nVdash \phi$. 


\begin{claim}\label{would-free}
    Suppose that $\mathcal{K},-\infty,g \nVdash \phi$. Then there is some $>$-free formula $\phi^\star$ such that $[\phi]^g=[\phi^\star]^g$.
\end{claim}

\begin{proof}
    Let's assemble three facts. 
    First, we are assuming that $\phi$ is false at $-\infty$, i.e., 
    \begin{align*}
        \mystar \quad [\phi]^g=[\phi]^g\cap\mathbb{Z}^-.
    \end{align*}
    Second, 
    by Claim~\ref{claim:reduceZ-}, there is a $\would$-free formula $\phi^-$ such that:
    \begin{align*}
        \mystar\mystar \quad [\phi]^g \cap \mathbb{Z}^- = [\phi^-]^g \cap \mathbb{Z}^-.
    \end{align*}
    Third, observe that:
    \begin{align*}
        \mystar\mystar\mystar \quad [\exists x F(x)]^g = \mathbb{Z}^-.
    \end{align*}
    Now let $\phi^\star \coloneqq \phi^- \wedge \exists x F(x)$. We reason as follows:
    \begin{align*}
        [\phi^\star]^g & = [\phi^- \wedge \exists x F(x)]^g \text{ by definition of $\phi^\star$}\\
        & =  [\phi^-]^g \cap [\exists x F(x)]^g \text{ by the semantics for $\wedge$}\\
        & =  [\phi^-]^g \cap \mathbb{Z}^- \text{ by $\mystar\mystar\mystar$}\\
        & =  [\phi]^g \cap \mathbb{Z}^- \text{ by $\mystar\mystar$}\\
        & =  [\phi]^g \text{ by $\mystar$}
    \end{align*}
    Since $\phi^\star$ is $\would$-free, this completes the proof of the claim.
\end{proof}

Throughout, we'll say $\phi$ is \emph{$\mathcal{K}$-equivalent} to $\phi'$ (relative to $g$) to mean $[\phi]^g = [\phi']^g$. 

\begin{claim}\label{into-dnf}
    Suppose that $\mathcal{K},-\infty,g \nVdash \phi$. Then $\phi$ is $\mathcal{K}$-equivalent to a disjunction in which each disjunct has of one of the following forms ($n,m \geq 0$):
    \begin{align*}
        & F(x_1) \wedge \cdots \wedge F(x_n) \wedge \neg F(y_1) \wedge \cdots \wedge \neg F(y_m) \\
        & F(x_1) \wedge \cdots \wedge F(x_n) \wedge \exists x \neg F(x) \\
        & \neg F(x_1) \wedge \cdots \wedge \neg F(x_n) \wedge \exists x F(x) \\
        & \exists x F(x) \wedge \exists x \neg F(x) \\
        & \forall x F(x)
    \end{align*}
\end{claim}

\begin{proof}
    By Claim~\ref{would-free}, we may assume that $\phi$ is $\would$-free. 
    $F$ is the only non-empty predicate in the language, so we may assume that any other atomic formula in $\phi$ has been replaced with $\bot$. 
    
    This is all to say that we may assume $\phi$ is a simple monadic formula with only the predicate $F$ appearing. 
    Such formulas can be written as Boolean combinations of formulas of the form $F(x)$, $\forall x F(x)$, and $\exists x F(x)$. 
    We can then rewrite this Boolean combination in DNF and eliminate redundant conjuncts. 
    That is, we may assume $\phi$ is a disjunction where each disjunct is a conjunction of formulas of the following forms:
    \begin{align*}
        & F(x_1) \wedge \cdots \wedge F(x_n)\\
        & \neg F(x_1) \wedge \cdots \wedge \neg F(x_n)\\
        & \exists x F(x) \\
        & \exists x \neg F(x) \\
        & \forall x F(x) \\
        & \forall x \neg F(x)
    \end{align*}
    Some such conjunctions are logically redundant, e.g., $F(x_1)\wedge \exists x F(x)$. 
    Others are logically contradictory, e.g., $F(x_1)\wedge \forall x \neg F(x)$. 
    Note also that $\forall x \neg F(x)$ can always be eliminated as a disjunct since $\phi$ is false at $-\infty$. 
    In a case-by-case fashion, one can remove these redundancies, so that $\phi$ is reduced to a disjunction in the exact form prescribed by the statement of the claim.
\end{proof}

Recall that we are assuming that our variable assignment $g$ is surjective (Remark \ref{surjective}). 
So for each $j\in\mathbb{Z}^-$, there is a formula $F(x)$ where $g(x) = j$, which we write as ``$F(j)$''. For the remaining claims, we will repeatedly apply the following observation.

\begin{claim}\label{observation}
    If $m\geq n$ then $\mathcal{K} \Vdash F(m)\then F(n)$.
\end{claim}

\begin{proof}
    Consult Figure 1. 
\end{proof}

\begin{claim}\label{boundary}
    Suppose that $\mathcal{K},-\infty,g \nVdash \phi$. Then $\phi$ is $\mathcal{K}$-equivalent to a disjunction in which each disjunct has one of the following forms:
    \begin{align*}
        & F(j) \wedge \neg F(j+1) \text{ where } k \leq j <-1 \\
        & \forall x F(x) \\
        & \neg F(k) \wedge \exists x F(x)
    \end{align*}
\end{claim}

\begin{proof}
    We assume that $\phi$ is in the form provided by Claim~\ref{into-dnf}. 
    We now consider the different disjuncts of $\phi$ case by case. 
    Note that $\forall x F(x)$ is already in the appropriate form. 
    We now consider the other cases.
    
    \textbf{Case 1:} Consider some instance of:
    $$(B) \quad F(x_1)\wedge\dots\wedge F(x_n) \wedge \neg F(y_1) \wedge \dots \wedge \neg F(y_m)$$
    Without loss of generality, let's stipulate that $g(x_1)$ is greatest among the $g(x_i)$s and $g(y_1)$ is least among the $g(y_i)$s.
    
    As a first step, we note that $(B)$ is $\mathcal{K}$-equivalent to:
    $$(C) \quad F(x_1)\wedge \neg F(y_1)$$
    That $(B)$ $\mathcal{K}$-entails $(C)$ is obvious. 
    For the other direction, note first that $g(x_1)>g(x_i)$ for all $i\neq 1$. 
    Hence Claim~\ref{observation} implies that the conjunct $F(x_1)$ implies each $F(x_i)$. 
    Using Claim~\ref{observation} and the the fact $g(y_1)<g(y_i)$, we similarly see that the conjunct $\neg F(y_1)$ implies each $\neg F(y_i)$. 
    
    We now want to see that $(C)$ is $\mathcal{K}$-equivalent to a disjunction of claims of the form $F(j)\wedge\neg F(j+1)$. 
    Note that $(C)$ states that---relative to the variable assignment $g$---the boundary between $F$ and $\neg F$ occurs within the interval $[x_1,y_1]$. 
    Hence $(C)$ $\mathcal{K}$-entails the disjunction:
    $$(D) \quad \big(F(x_1)\wedge\neg F(x_1+1)\big)\vee \dots \vee \big(F(y_1-1)\wedge \neg F(y_1)\big)$$
    Appealing to Claim~\ref{observation} once again, we see that each disjunct of $(D)$ $\mathcal{K}$-entails $(C)$. 
    All of this is to say that $(B)$ is $\mathcal{K}$-equivalent to $(D)$.
    
    \textbf{Case 2:} Now a formula of the form 
    $$(E) \quad F(x_1)\wedge\dots\wedge F(x_n) \wedge \exists x \neg F(x)$$
    is $\mathcal{K}$-equivalent to a disjunction of claims of the form $F(j) \wedge \neg F(j+1)$. 
    For this we use the same reasoning as in the previous case. 
    We once again stipulate that $g(x_1)$ is greatest among the $g(x_i)$s. 
    Then, by applying Claim~\ref{observation} once again, we see that $(E)$ is $\mathcal{K}$-equivalent to the disjunction:
    $$\big(F(x_1)\wedge\neg F(x_1+1)\big)\vee \dots \vee \big(F(-2)\wedge \neg F(-1)\big)$$
    
    \textbf{Case 3:} Finally, let's consider formulas of the form:
    $$(F) \quad \neg F(x_1)\wedge\dots\wedge \neg F(x_n) \wedge \exists x F(x)$$
    Without loss of generality, let's stipulate that $g(x_1)$ is least among the $g(x_i)$s. 
    Then Claim~\ref{observation} immediately entails that $(F)$ is $\mathcal{K}$-equivalent to:
    $$\neg F(x_1)\wedge \exists x F(x)$$
    This completes the proof.
\end{proof}

We have assembled all the claims needed to finally show that $\mathsf{CEM}$ is valid in the order model $\mathcal{K}$.

\begin{claim}\label{with-cem}
    $\mathcal{K} \Vdash \mathsf{CEM}$.
\end{claim}


\begin{proof}
    By Claim~\ref{claim:reduceZ-}, $\mathcal{K},k,g \Vdash \mathsf{CEM}$ for $k\in \mathbb{Z}^-$. 
    Hence, it suffices to show that $\mathcal{K},-\infty\Vdash (\phi \would \psi) \vee (\phi \would \neg \psi)$. 
    This is trivial if $\mathcal{K},-\infty \Vdash \phi$, so suppose that $\mathcal{K},-\infty \nVdash \phi$. 
    
    Since $\exists x F(x)$ is false at $-\infty$, so is $\psi\wedge \exists x F(x)$. 
    Hence, Claim~\ref{boundary} entails that $\psi\wedge \exists x F(x)$ is $\mathcal{K}$-equivalent to some formula $\psi^\star$ that is a disjunction of particular type. 
    Since $\phi$ is false at $-\infty$, this implies that: $$\mathcal{K},-\infty\Vdash (\phi \would \psi) \equiv (\phi \would \psi^\star).$$
    
    By Claim~\ref{boundary}, we may also assume that $\phi$ is a disjunction of a particular form. Note that $[\phi]^g$ \emph{might} or \emph{might not} have a minimal element. 
    Indeed, if any of
    $$\neg F(k) \wedge \exists x F(x) \quad  \neg F(k) \quad  \exists x F(x) \quad \exists x F(x) \wedge \exists x \neg F(x)$$ 
    is a disjunct of $\phi$, then $[\phi]^g$ has no minimal element. 
    Otherwise, it does. 
    We now consider cases depending on whether $[\phi]^g$ has a minimal element.
    
    \emph{Case 1:} Suppose that $[\phi]^g$ has a minimal element. 
    Then $(\phi \would \psi) \vee (\phi \would \neg\psi)$ holds at $-\infty$.
    
    \emph{Case 2:} Suppose that $[\phi]^g$ does not have a minimal element. We consider subcases:
    
    \emph{Case 2.1:} Suppose that, for some $j\geq k$, any one of:
    $$\neg F(j) \wedge \exists x F(x) \quad \neg F(j) \quad \exists x F(x) \quad \exists x F(x) \wedge \exists x \neg F(x)$$ 
    is a disjunct of $\psi^\star$. 
    Since $j\geq k$, Claim~\ref{observation} entails that $\neg F(k)\then \neg F(j)$ is a global $\mathcal{K}$-truth. 
    Hence, $\phi \would \psi^\star$ holds at $-\infty$. 
    Therefore $\phi \would \psi$ holds at $-\infty$ as well. 
    
    \emph{Case 2.2:} Suppose that we are not in Case 2.1. 
    There are only a few options left. 
    One option is that there is a $j < k$ such that one of:
    $$\neg F(j) \wedge \exists x F(x) \quad \neg F(j)$$
    is a disjunct of $\psi^\star$. 
    The other option is that $[\psi]^g$ has a minimal element. 
    Either way, $\phi \would \neg\psi^\star$ holds at $-\infty$, whence $\phi \would \neg \psi$ holds at $-\infty$ as well.
\end{proof}

Claims \ref{descending sequences}, \ref{without-cem}, and \ref{with-cem} jointly entail Proposition \ref{counter-model}. This yield the main theorem of the section:

\begin{theorem}
\label{thm:QC2-incomp}
    $\mathsf{QC2}$ is not the logic of any class of selection frames.
\end{theorem}

\begin{proof}
    Let $\mathcal{C}$ be a class of selection function frames. 
    If $\mathcal{C}$ contains a frame $\mathcal{F}$ that is not weakly Stalnakerian, then, by Proposition~\ref{Proposition:FrameCorrespondence}, $\mathsf{QC2}$ is not sound with respect to $\mathcal{C}$. 
    On the other hand, if $\mathcal{C}$ contains only weakly Stalnakerian frames, then $\neg \mathsf{DS}$ is valid over $\mathcal{C}$ by Proposition \ref{prop:DS}. 
    However, by Proposition \ref{counter-model}, $\neg\mathsf{DS}$ is not a theorem of $\mathsf{QC2}$, whence $\mathsf{QC2}$ is not complete with respect to $\mathcal{C}$.
\end{proof}

\section{Refinements}
\label{Section:Refinements}

In this section, we briefly show how the incompleteness proof for $\mathsf{QC2}$ (Theorem~\ref{thm:QC2-incomp}) can be generalized in various ways. 
In particular, we show that the result can be extended to languages with $=$ and to logics with variable domains. 
And since one of these logics ($\mathsf{QC2^v_=}$) is provably equivalent to $\mathsf{QST}$, this means that the incompleteness result applies to $\mathsf{QST}$, that is, the conditional logic originally studied by \cite{StalnakerThomason1970}. 

\subsection{The Logic with Identity}
\label{Subsection:LogicwithIdentity}

Let's first consider how the proof of Theorem~\ref{thm:QC2-incomp} can be extended to $\mathsf{QC2_=}$. 
Our target theorem is:

\begin{theorem}
\label{thm:QC2=-incomp}
    $\mathsf{QC2_=}$ is not the logic of any class of selection frames.
\end{theorem}




Proposition \ref{Proposition:FrameCorrespondence} already applies to $\mathsf{QC2_=}$, so we can appeal to it as before. 
Moreover, Proposition \ref{prop:DS} does not turn on the presence or absence of $=$. 
Hence, it suffices to prove the following analogue of Proposition~\ref{counter-model}:

\begin{proposition}\label{new-counter-model}
    $\mathsf{QC2_=} \nvdash \mathsf{DS}$.
\end{proposition}

The proof of Proposition \ref{new-counter-model} is quite similar to the proof of Proposition \ref{counter-model}.
Indeed, the strategy is the same: Construct an order model $\mathcal{K}$ of $\mathsf{QC2_=}+\neg\mathsf{DS}$. 
The primary difference is in how we show that  $\mathcal{K} \Vdash \mathsf{CEM}$. 
In the proof of Proposition \ref{counter-model}, we relied on the fact that if $\mathcal{K},-\infty,g \Vdash \phi$, then $\phi$ is $\mathcal{K}$-equivalent to a formula in a very special form. 
This reduction is carried out in Claims~\ref{would-free}--\ref{boundary}.
The reduction becomes more complicated---although only slightly---when $=$ is included in the language.


\begin{proof}
We use the same model $\mathcal{K}$ from Definition \ref{order-model}. 
As before, it is easy to verify the validity of all axioms and rules with the exception of CEM. 
And, as before, the only issue is showing that $(\phi \would \psi) \vee (\phi \would \neg\psi)$ holds at the point $-\infty$ in the case when $\mathcal{K},-\infty,g\nVdash\phi$.

The following analogue of Claim~\ref{would-free} holds using the exact same proof:

\begin{claim}\label{new-would-free}
    Suppose that $\mathcal{K},-\infty,g \nVdash \phi$. Then there is some $>$-free formula $\phi^\star$ such that $[\phi]^g=[\phi^\star]^g$.
\end{claim}

In the proof of Proposition \ref{counter-model}, we then proceeded to show that $\phi$ could be put into disjunctive normal form wherein each disjunct is extremely simple. This was facilitated by the fact that the monadic predicate $F$ was the only non-vacuous predicate, and there are very few things one can express with only one non-vacuous monadic predicate. 
With the identity symbol in the signature, we can say slightly more. 
We can now state that some things are identical or not identical and we can now \emph{count} how many non-$F$s there are. 
But, as we'll see, we effectively already had the ability to ``count'' non-$F$s.

Before continuing, we introduce some new notation. 
For a formula $\psi(x)$ and a number $n>0$, we introduce the following abbreviation:
\begin{align*}
    \exists_{=n}x \ \psi(x) \coloneqq \exists x_1\dots\exists x_n \Big( \bigwedge_{i\leq n}\psi(x_i) \wedge \forall y \big(\psi(y) \to \bigvee_{i\leq n}y=x_i\big) \wedge \bigwedge_{i\neq j \leq n} x_i\neq x_j\Big).
\end{align*}


We can now state an analogue of Claim~\ref{into-dnf}.

\begin{claim}\label{new-into-dnf}
    Suppose that $\mathcal{K},-\infty,g \nVdash \phi$. 
    Then $\phi$ is $\mathcal{K}$-equivalent to a disjunction in which each disjunct consists of some identity and non-identity claims conjoined with one of the following ($n,m \geq 0$, $k \geq 1$):
    \begin{align*}
        & F(x_1) \wedge \cdots \wedge F(x_n) \wedge \neg F(y_1) \wedge \cdots \wedge \neg F(y_m) \wedge \exists_{=k} x \ \neg F(x)\\
        & F(x_1) \wedge \cdots \wedge F(x_n) \wedge \exists x \neg F(x) \\
        & \neg F(x_1) \wedge \cdots \wedge \neg F(x_n) \wedge \exists x F(x) \wedge \exists_{=k} x \ \neg F(x)\\
        & \exists x F(x) \wedge \exists x \neg F(x) \\
        & \forall x F(x)
    \end{align*}
\end{claim}

\begin{proof}
    By Claim~\ref{new-would-free}, we may assume that $\phi$ is $\would$-free. 
    $F$ and $=$ are the only non-empty predicates in the language, so we may assume that any other atomic formula in $\phi$ has been replaced by $\bot$. 
    The rest of the proof proceeds as before, where we make use of the following observation: 
    If there is a single $F$, then there are infinitely many. 
    Hence, the maximally informative quantitative claims one can make are $\forall x F(x)$ (i.e., there are no non-$F$s) and $\exists_{=k} x \ \neg F(x)$ (i.e., there are exactly $k$ $\neg F$s).
\end{proof}

This yields the following analogue of Claim~\ref{boundary}:

\begin{claim}\label{new-boundary}
    Suppose that $\mathcal{K},-\infty,g \nVdash \phi$. 
    Then $\phi$ is $\mathcal{K}$-equivalent to a disjunction in which each disjunct consists of some identity and non-identity claims conjoined with one of the following:
    \begin{align*}
        & F(j) \wedge \neg F(j+1) \text{ where } k \leq j <-1 \\
        & \forall x F(x) \\
        & \neg F(k) \wedge \exists x F(x)
    \end{align*}
\end{claim}

\begin{proof}
    The reduction is no different here. 
    The only important observation is that the statement $\exists_{=n} x \ \neg F(x)$ is $\mathcal{K}$-equivalent to the statement $F(-n-1)\wedge \neg F(-n)$. 
    Moreover, we can state every sentence of the form $F(-n-1)\wedge \neg F(-n)$ since we are assuming the surjectivity of our variable assignment (Remark~\ref{surjective}).
\end{proof}

The final claim now follows easily.

\begin{claim}\label{new-with-cem}
    $\mathcal{K} \Vdash \mathsf{CEM}$.
\end{claim}

\begin{proof}
    We need only modify the proof of the analogous Claim~\ref{with-cem}. 
    Note that via Claim~\ref{new-boundary}, we may reduce $\phi$ into a special form. 
    The disjuncts in $\phi$ may now have some identity and non-identity claims appearing as conjuncts within them. 
    However, relative to any variable assignment $g$, all such claims are either true at all worlds or at none (that is, variables are interpreted rigidly). 
    The true (non-)identity claims are intersubstitutable with $\top$, and the false ones are intersubstitutable with $\bot$. 
    Thus, in neither case does their presence change the proof.
\end{proof}
This completes the proof of Proposition~\ref{new-counter-model}.
\end{proof}

\subsection{Logics with Variable Domains}
\label{Subsection:LogicVariableDomains}

Generalizing the proof of Theorem~\ref{thm:QC2-incomp} to $\mathsf{QC2^v_{\ex}}$ and $\mathsf{QC2^v_=}$ is straightforward. 
First, neither Proposition~\ref{Proposition:FrameCorrespondence} nor Proposition~\ref{prop:DS} depend on any constraints on local domains. 
As for Propositions~\ref{counter-model} and \ref{new-counter-model}, these immediately apply to $\mathsf{QC2^v_{\ex/=}}$ since globally constant domain models are just a special case of variable domain models. 
(More directly, it is easy to verify that $\mathsf{QC2}+Ex \vdash \mathsf{QC2^v_{\ex}}$, and that $\mathcal{K} \Vdash Ex$.) 
Hence, all of these results carry over immediately to the variable domain case. 

\begin{theorem}
\label{thm:QC2v-incomp}
    $\mathsf{QC2^v_{\ex/=}}$ is not the logic of any class of selection frames. 
    Similarly for $\mathsf{QC2^c_{\ex/=}}$. 
\end{theorem}


Since, as we have already observed, $\mathsf{QC2^v_{=}}$ is provably equivalent to $\mathsf{QST}$, we have the following notable corollary of Theorem~\ref{thm:QC2v-incomp}:

\begin{corollary}
\label{Corollary:QST}
    $\mathsf{QST}$ is not the logic of any class of selection frames. 
\end{corollary}

\noindent Thus, our incompleteness result is quite general. It applies regardless of the presence or absence of $=$, and regardless of the constraints placed on local domains. 

\section{Concluding Remarks}
\label{Section:Concluding Remarks}


\cite{StalnakerThomason1970} famously showed that the quantified conditional logic $\mathsf{QST}$ is sound and complete for a class of selection frames in which selection functions take formulas (and variable assignments; see Definition~\ref{Definition:Quasi Selection Frame}) as arguments. 
As we observed in Section~\ref{Section:Introduction}, many of the philosophical applications Stalnaker, in particular, has wanted to make of this semantics require that selection functions instead take propositions as arguments. 
It is tempting to conjecture that Stalnaker and Thomason's completeness result does not hinge on this difference. 
In this paper, we showed that this conjecture is false: $\mathsf{QST}$ is not the logic of any class of selection frames in which selection functions are defined on (all) propositions. 

This shows that there is a sense in which the completeness result obtained by Stalnaker and Thomason is an accident. 
Their completeness proof crucially depended on subtle semantic design choices, specifically concerning what arguments selection functions can take. 
On the other hand, we have demonstrated a robust incompleteness theorem which does not crucially hinge on features of the language or what constraints are imposed on domains. 

To be clear, our incompleteness result does not establish that the logical problem of conditionals, as Stalnaker conceived it, is unsolvable. 
Our result only shows that a certain class of logics, including the one Stalnaker and Thomason put forward, cannot constitute solutions to this problem in the presence of first-order quantifiers. 
The question of whether, and to what extent, the logical problem of conditionals is solvable by other logics remains open. 
In closing, we highlight some remaining questions that our incompleteness result raises for the study of quantified conditional logic. 

First, we have not determined whether our incompleteness result applies to conditional logics in which the logic of the outer modality ($\nec$) is strengthened. 
It is notable that the reductions utilized in the proof of Proposition~\ref{counter-model} all exploit the fact that the accessibility relation of the model isn't universal (specifically, $R(k) = \{k\}$; see Definition~\ref{order-model}). 
The accessibility relation does obey some nontrivial constraints, such as reflexivity and transitivity, and so we should expect incompleteness to apply even to logics as strong as $\mathsf{S4}$ for the outer modality. 
But it is an open question whether our incompleteness result can be generalized to conditional logics with an outer modal logic such as $\mathsf{S5}$.\footnote{Conditional logics with an $\mathsf{S5}$ outer modal logic include the systems \cite[pp.~137--138]{Lewis1973} calls \textsf{VCU} and \textsf{VCSU}.} It perhaps would not be \textit{too} surprising if having an $\mathsf{S5}$ outer modal logic made some difference; \cite{Fine1970propquantifiers} showed that whereas \textsf{S5} with quantifiers ranging over all propositions, $\mathsf{S5}\pi+$, is decidable, for many other modal logics \textsf{L}, $\mathsf{L}\pi+$ is not even axiomatizable. 

Second, our incompleteness theorems relied on showing that a certain formula, $\neg\mathsf{DS}$, fails to be a theorem of many quantified conditional logics even though it is valid over the relevant classes of frames. 
We have not determined whether the addition of $\neg\mathsf{DS}$ \textit{as an axiom} to these systems might yield systems that are complete. 
If so, a philosophical assessment of the desirability of $\neg\mathsf{DS}$ would be called for. 

Finally, there are many natural classes of selection frames such that we do not know whether the logic of the class is axiomatizable. In particular, we have not determined whether $\mathsf{L}(\mathcal{C})$ is axiomatizable where $\mathcal{C}$ is the class of all globally constant (weakly) Stalnakerian frames. 
If this class should turn out not to be axiomatizable, this would have significant ramifications for the feasibility of various philosophical projects Stalnaker and others intended this semantic framework for. 

Either affirmative or negative answers to any of these questions would carry philosophical significance, as they would shed light on the extent to which the solution to/solvability of the logical problem of conditionals is contingent on various choice points in how one develops a quantified conditional logic. 
If frame completeness can be restored by making certain substantive assumptions about accessibility or by adding new axioms, that would make the philosophical assessment of these assumptions and axioms more pressing. 
If, on the other hand, frame incompleteness was fairly robust to such semantic design choices, that would suggest that the project of solving the logical problem of conditionals was doomed from the start. 
This is an area that demands further investigation.

\printbibliography

\end{document}